\documentclass[12pt,a4paper,reqno]{amsart}

\usepackage{amssymb}
\usepackage{amsmath}
\usepackage[mathscr]{euscript}
\usepackage{times}
\usepackage{newcent}
\usepackage{xcolor}
\usepackage{graphicx}
\usepackage{hyperref}
\usepackage[top=30truemm,bottom=30truemm,left=25truemm,right=25truemm]{geometry}
\usepackage{tikz}
\usepackage{comment}

\makeatletter
\@addtoreset{equation}{section}
\makeatother

\newtheorem{theorem}{Theorem}[section]
\newtheorem{lemma}[theorem]{Lemma}
\newtheorem{proposition}[theorem]{Proposition}
\newtheorem{corollary}[theorem]{Corollary}

\newtheorem{remark}[theorem]{Remark}

\newcommand{\mc}[1]{{\mathcal #1}}
\newcommand{\mf}[1]{{\mathfrak #1}}

\newcommand{\bb}[1]{{\mathbb #1}}

\newcommand{\e}{\epsilon}

\newcommand{\lan}{\langle}
\newcommand{\ran}{\rangle}

\newcommand{\N}{\bb N}
\newcommand{\Z}{\bb Z}

\newcommand{\R}{\bb R}

\newcommand{\T}{\bb T}

\newcommand{\Lip}{{\rm Lip}}

\newcommand{\supp}{{\rm supp}\,}
\newcommand{\ac}{{\rm ac}}

\renewcommand{\tilde}{\widetilde}

\begin{document}

\title[LDP for ZRP]{Large deviations for additive functionals of one-dimensional zero-range processes}
\author[I. Seo]{Insuk Seo}
\address{Department of Mathematical Sciences and Research Institute of Mathematics, Seoul National University, Republic of Korea.}
\email{insuk.seo@snu.ac.kr}

\author[K. Tsunoda]{Kenkichi Tsunoda}
\address{Faculty of Mathematics, Kyushu University, Japan.}
\email{tsunoda@math.kyushu-u.ac.jp}
\date{\today}

\maketitle

\begin{abstract}
We prove a large deviation principle for additive functionals of local Lipschitz functions of a one-dimensional symmetric zero-range process on the diffusive scale. The rate function is obtained from the dynamical large deviation principle for the empirical density and is characterized by a variational constraint involving the density at a fixed spatial point. On the microscopic side, a super-exponential replacement is needed to replace the local observable with a function of the empirical density. On the macroscopic side, the spatial regularity enjoyed by finite-rate trajectories is needed to pass from spatial averages to point evaluation.
\end{abstract}

\section{Introduction}

Additive functionals are basic objects in probability theory. They represent the accumulated value of an observable along the evolution of the process and appear in many problems concerning fluctuations and large deviations. Classical contributions include the long-time large deviation theory of Donsker and Varadhan \cite{dv75} and the central limit theorem for additive functionals of reversible Markov processes due to Kipnis and Varadhan \cite{kv86}.

In systems with many interacting particles, one can look at the system on two different spatial scales. A local observable describes what happens near a fixed site, while the empirical density captures the bulk behavior of the system on the macroscopic scale. For many conservative particle systems, the empirical density converges, under diffusive scaling, to the solution of a hydrodynamic equation. Large deviations from this macroscopic evolution are also well understood in several models, including exclusion and zero-range processes and more general conservative lattice gases \cite{kov89,bkl95,qrv99,bfg13,fg23,fgh25}.

It is then natural to ask whether the large deviation principle for the empirical density can be used to study a local additive functional. The occupation time of a fixed site is a basic example. More generally, let $h$ denote an observable which depends only on the configuration in a fixed finite region. If the system is locally close to equilibrium with density $r$, let $H(r)$ denote the equilibrium expectation of $h$ at that density. One expects the time integral of $h$ to be approximated, on the macroscopic scale, by the time integral of $H$ evaluated along the macroscopic density at the observed point.

There is, however, a difficulty in making this argument rigorous. The dynamical large deviation principle for the empirical density is formulated in a weak topology, which controls spatial averages of the density but not its value at a prescribed point. Hence the point-evaluation functional is not continuous, and the contraction principle cannot be applied directly. To overcome this obstruction, we approximate the density at the observed point by averages over shrinking neighborhoods and use the spatial regularity of finite-rate trajectories to remove this averaging. We answer this question for the one-dimensional symmetric zero-range process on a discrete torus with $N$ sites, where $N$ denotes the system size. Our main contribution is to show how the dynamical large deviation principle for the empirical density can be used to derive a large deviation principle for microscopic local additive functionals.

More precisely, for every local Lipschitz function $h$, we prove a large deviation principle, on the diffusive scale and with speed $N$, for its time integral. To describe the rate function, write $\rho(t,\theta)$ for the macroscopic density at time $t$ and position $\theta$, and let $I_\alpha$ denote the dynamical large deviation rate function when the process starts from equilibrium at density $\alpha$. Theorem \ref{thm5-1} shows that the rate function at $a$ is obtained by minimizing $I_\alpha$ over density trajectories satisfying the constraint $\int_0^T H(\rho(t,0))dt=a.$ When $h$ is the number of particles at the origin, this gives the large deviation principle for the occupation time of that site.

The proof has two main steps. First, on the macroscopic side, the spatial regularity of finite-rate trajectories allows us to approximate the density at a fixed point by spatial averages, uniformly on level sets of the dynamical rate function. Second, on the microscopic side, a super-exponential replacement estimate allows us to replace the local observable by a quantity determined by the empirical density. An approximation argument based on the contraction principle can then be applied to the regularized functional, and the spatial averaging is removed using the regularity of finite-rate trajectories. In this way, the large deviation principle for the empirical density yields the same result for the local additive functional.

For the zero-range process, an additional difficulty arises because a local Lipschitz function is unbounded in general. A local Lipschitz observable need not be bounded, because arbitrarily many particles may occupy a single site. To overcome this difficulty, Theorem \ref{thm5-2} establishes the required super-exponential replacement estimate. We first cut off configurations with high local density, prove the replacement estimate after this cutoff, and then show that the discarded contribution is super-exponentially small.

The restriction to one dimension is intrinsic to the problem, and this dimensional dependence appears on both the microscopic and macroscopic sides. For the occupation time of the symmetric simple exclusion process, it is known that the large deviation speeds over a microscopic time interval of length $t$ are $\sqrt{t}$ in dimension one, $t/\log t$ in dimension two, and $t$ in dimensions three and higher \cite{lan92,cll04}. Under diffusive scaling, these correspond to $N$, $N^2/\log N$, and $N^2$, respectively. Note that the dynamical large deviation principle for the empirical density has speed $N^d$ in dimension $d$. Therefore, the two scales coincide only in dimension one.

There is also an analytic reason for this dimensional dependence: in one dimension, finite-rate trajectories have enough spatial regularity, through a Sobolev-type estimate, to recover the density at a point from averages over shrinking intervals. In dimension two, a different spatial structure appears; for the exclusion process it is described by the polar empirical measure of Chang, Landim and Lee \cite{cll04,cll19}. Thus, higher-dimensional occupation-time problems are not direct extensions of the one considered here.

We now briefly situate our result in the existing literature, organized by the scale of deviations. At the fluctuation scale, additive functionals of interacting particle systems have been studied extensively, beginning with the work of Kipnis and Varadhan \cite{kv86} and followed by, among others, Sethuraman and Xu \cite{sx96}, Quastel, Jankowski and Sheriff \cite{qjs02}, and Gon\c{c}alves and Jara \cite{gj13}. At the intermediate moderate-deviation scale, Gao and Quastel obtained results for stochastic lattice gases and occupation times of the symmetric exclusion process \cite{gq24,gq25}. By contrast, large deviations for local additive functionals are much less developed. For the symmetric simple exclusion process, the main existing results concern occupation times, including the works of Landim \cite{lan92} and Chang, Landim and Lee \cite{cll04,cll19}. For independent particle systems, Benois \cite{ben96} studied large deviations for occupation times. In a related direction, Avena, Jara and V\"{o}llering established a joint path large deviation principle for a slowed random walk driven by the symmetric simple exclusion process \cite{ajv18}.

Our result applies to additive functionals of general local Lipschitz functions, rather than only to a particular occupation-time observable. The argument is not limited to gradient systems: for one-dimensional nongradient exclusion processes, the available dynamical large deviation principles and replacement estimates lead to the same type of result \cite{fuy95,vy97,bfg13}. What is specific to the zero-range process is the additional difficulty caused by unbounded local functions, which requires the high-density cutoff developed in this paper.

The paper is organized as follows. Section \ref{sec:model} introduces the model and states the main results. Section \ref{sec:ratefunction} identifies the rate function and proves the energy estimates used to control the density at a fixed point. Section \ref{sec:replacement} proves the super-exponential replacement theorem. Section \ref{sec:proof} completes the proof of the large deviation principle for the additive functional.

We use the following notation throughout the paper.
\begin{itemize}
\item We write $\mathbb{N}:=\{0,1,\ldots\}$, $\mathbb{R}_+:=[0,\infty)$,
and use $\mathbb{Z}$ for the set of integers.
\item The symbols $C,C_0,C_1,\ldots$ denote positive constants whose values may change
from line to line. Unless otherwise specified, these constants are independent of the parameters under consideration.
\item We write $\mathbf{1}\{A\}$ for the indicator function of an event or a condition $A$.
\item Microscopic positions are denoted by $x,y,z$, while macroscopic positions are denoted by $\theta,\theta_1,\theta_2$.
\item The symbol $\alpha>0$ denotes the fixed reference density, while $\rho$ denotes a generic density or density profile.
\item The operators $\partial_t$, $\nabla$, and $\Delta$ denote the time derivative, the spatial gradient, and the spatial Laplacian, respectively.
\item For a configuration $\eta$, $\tau_x\eta$ denotes its spatial shift: $(\tau_x\eta)(y):=\eta(x+y)$.
\item When $\epsilon N$ or $\varepsilon N$ appears as a block size or as the upper limit of a sum, we take its integer part.
\item For quotients involving the density, we use the convention
$0/0:=0$ and $a/0:=\infty$ for $a>0$.
\end{itemize}

\newpage

\section{Model and results}\label{sec:model}

In this section, we introduce our model and main result.

\subsection{Zero-range process}

Fix a positive integer $N$ and let $\T_N:=\Z/N\Z$ be a one-dimensional discrete torus.
Generic elements of $\T_N$ are denoted by $x,y,z$ and sums in $\T_N$ are taken modulo $N$.
Let $X_{N}:=\mathbb{N}^{\mathbb{T}_{N}}$. We write $\eta:=(\eta(x))_{x\in\mathbb{T}_{N}}\in X_{N}$
for a particle configuration on $\mathbb{T}_{N}$ where site $x$ contains $\eta(x)$ particles.

The zero-range process $(\eta_{t})_{t\ge0}:=(\eta_{t}^{N})_{t\ge0}$
on $\mathbb{T}_{N}$ is a continuous-time Markov chain on $X_{N}$
whose generator is given by
\begin{align}\label{2-4}
( L_{N}f)(\eta):=N^2\sum_{x, y\in\mathbb{T}_{N}}g(\eta(x))P(x,y)\left[f(\eta^{x,y})-f(\eta)\right], 
\end{align}
where $\eta^{x,y}$ stands for the configuration obtained after a jump of a particle from $x$ to $y$ and hence defined only when $\eta(x)\ge 1$:
\begin{align*}
\eta^{x,y}(z):=
\begin{cases}
\eta(x)-1,\quad \text{if $z=x$},\\
\eta(y)+1,\quad \text{if $z=y$},\\
\eta(z),\quad \text{otherwise}.\\
\end{cases}
\end{align*}
Here $P(\cdot,\cdot)$ is the transition kernel for the underlying
random walk constituting the particle system, and $g:\N\to\R_+$
is the function governing the zero-range interaction between particles.
For simplicity, we restrict attention to the symmetric nearest-neighbor case,
that is, $P(x,y):=1/2$ if  $|x-y|=1$, $P(x,y):=0$, otherwise.
Then the generator $L_N$ can be rewritten as
\begin{align*}
( L_{N}f)(\eta)=\dfrac{N^2}{2}\sum_{x\in\mathbb{T}_{N}}\sum_{e=\pm1}g(\eta(x))\left[f(\eta^{x,x+e})-f(\eta)\right].
\end{align*}
We note that it is straightforward to generalize our results
to mean-zero, irreducible, translation-invariant, and finite-range cases.

We shall assume the following conditions on the jump rate $g$ throughout this article:
$g(0)=0, g(k)>0$ for any $k=1,2,\ldots$ and
\begin{align}
&g^*:=\sup_{k\in\N}|g(k+1)-g(k)|<\infty, \tag{LG}\label{LG}\\
&\exists k_0\in\N, c_*>0 \quad \text{such that} \quad g(k_1)-g(k_2)\ge c_*, \quad \forall k_1\ge k_2+k_0. \tag{M}\label{M}
\end{align}
Then, under \eqref{LG} and \eqref{M}, there exist constants
$c_0, C_0>0$ such that for any $k\in\N$,
\begin{align}\label{2-1}
c_0k \le g(k) \le C_0 k.
\end{align}

Define the partition function $Z:\mathbb{R}_{+}\to\mathbb{R}_{+}$
as
\begin{equation}\label{2-2}
Z(\varphi):=\sum_{k=0}^{\infty}\frac{\varphi^{k}}{g(k)!}, \quad \varphi\in \R_+,
\end{equation}
where $g(k)!$ is a shorthand for $g(1)\times\cdots\times g(k)$ with
the convention $g(0)!:=1$. By \eqref{2-1}, we can verify that the
radius of convergence of the power series \eqref{2-2} is $\infty$.
Therefore, the partition function $Z$ defines a smooth, increasing function
from $\mathbb{R}_{+}$ to $\mathbb{R}_{+}$.

For each $\varphi\in\R_+$, define $\widetilde{\nu}_{\varphi}^{N}$ to be
the translation-invariant product measure on $X_N$
with marginal
\[
\widetilde{\nu}_{\varphi}^{N}(\eta:\eta(x)=k):=\frac{1}{Z(\varphi)}\frac{\varphi^{k}}{g(k)!}, \quad  k\in\mathbb{N} .
\]
Then, for all $\varphi \ge 0$, the measure $\widetilde{\nu}_{\varphi}^{N}$
is an invariant measure for the zero-range process $(\eta_{t})_{t\ge0}$.
Now we define
\[
R(\varphi):= \int \eta(0) \widetilde\nu_\varphi^N(d\eta), \quad \varphi\in \R_+,
\]
so that $R(\varphi)$ denotes the particle density corresponding to
$\widetilde{\nu}_{\varphi}^{N}$. Then, we can readily verify that
$R:\mathbb{R}_{+}\to\mathbb{R_{+}}$ is a strictly increasing
bijection, and hence has an inverse function $\Phi:\R_+\to\R_+$.
Therefore, we can parametrize the invariant measure by $\rho$: $\nu_{\rho}^{N}:=\widetilde{\nu}_{\Phi(\rho)}^{N}$.
For simplicity, we sometimes denote $\nu_\rho^N$ by $\nu_\rho$.
For a measure $\mu$ on $X_N$, we denote expectation with respect to $\mu$ by $E_\mu$.
Then, one can readily deduce that $\Phi(\rho) = E_{\nu_\rho}\left[ g(\eta(0))\right], \rho\in\R_+$.
Thus, $\Phi$ is smooth on $\R_+$ and, by \eqref{2-1}
\begin{equation}\label{boundPhi}
c_0\rho\le\Phi(\rho)\le C_0\rho, \quad  \rho\in\mathbb{R}_{+}.
\end{equation}
In addition, under \eqref{LG} and \eqref{M} it is shown in \cite[(5.2)]{lsv96} that for some constants $c_1,C_1>0$,
\begin{align}\label{2-5}
c_1\rho\le E_{\nu_\rho}\left[\left(\eta(0)-\rho\right)^2 \right] \le C_1\rho, \quad \rho\in\R_+.
\end{align}
Moreover, since
\begin{align*}
\Phi'(\rho)=\dfrac{\Phi(\rho)}{E_{\nu_\rho}\left[\left(\eta(0)-\rho\right)^2 \right] }, \quad \rho>0,
\end{align*}
there exist constants $c_2, C_2>0$ such that
\begin{equation}\label{boundPhi'}
c_2\le \Phi'(\rho) \le C_2, \quad  \rho\in\mathbb{R}_{+}.
\end{equation}
In particular, $\Phi$ is Lipschitz continuous on $\R_+$.

Let $\T:=\R/\Z=[0,1)$ be a one-dimensional torus. A generic element of $\T$ is denoted by $\theta$.
Define the empirical measure corresponding to the zero-range process $(\eta_{t})_{t\ge0}$ by
\[
\pi_{t}^{N}:=\frac{1}{N}\sum_{x\in\mathbb{T}_{N}}\eta_{t}^N(x)\delta_{x/N},
\]
where $\delta_{\theta}$ represents the Dirac measure at $\theta\in\T$. We fix $T>0$
and regard $(\pi_{t}^{N})_{t\in[0,T]}$ as a Markov process on the
space of nonnegative Radon measures on $\mathbb{T}$ which will be denoted
by $\mc M_{+}(\mathbb{T})$, endowed with the weak topology.

The hydrodynamic limit for the zero-range process is known; see \cite[Theorem 5.1.1]{kl99}.
The hydrodynamic limit states that, under an appropriate assumption on the convergence of $\pi_0^N$ to
a given measurable function $\rho_0:\T\to\R_+$ as $N\to\infty$, $\pi_t^N$ converges as $N\to\infty$
to the unique weak solution to the Cauchy problem
\begin{align*}
\begin{cases}
\partial_t\rho=\dfrac12\Delta \Phi(\rho),\\
\rho(0,\cdot)=\rho_0(\cdot).
\end{cases}
\end{align*}

\subsection{Large deviations for empirical measure}

Define $\Omega_{T}:=D([0, T], \mc {M}_{+}(\mathbb{T}))$
as the space of all c\`adl\`ag trajectories taking values in $\mc M_+(\bb \T)$.
For $\pi\in\Omega_{T}$, we write $\pi_t$ or $\pi(t,\cdot)$ for the marginal of $\pi$ at time $t\in[0, T]$. 
Fix a reference density $\alpha>0$.
Let $\bb P_\alpha^N$ be the law of the process $(\eta_{t}^N)_{t\in[0,T]}$ on 
$D([0,T], X_N)$ started from the invariant measure $\nu_{\alpha}^{N}$,
and let $\bb Q_\alpha^N$ be the law of the process $(\pi_{t}^{N})_{t\in[0, T]}$ on $\Omega_T$.

Define $\Omega_T^\ac$ as the space of $\pi\in\Omega_{T}$ such that $\pi_{t}$ is absolutely
continuous with respect to the Lebesgue measure, denoted by $d\theta$, for all $t\in[0, T]$.
For $G\in C^{1, 2}([0, T]\times\mathbb{T})$ and  $\pi\in\Omega_T^\ac$ with $\pi_{t}(d\theta)= \rho(t, \theta)d\theta$
for $t\in[0, T]$, define
\begin{align*}
\mc J_G(\pi):= & \left\langle  \rho(T, \cdot), G(T, \cdot)\right\rangle -\left\langle  \rho(0, \cdot), G(0, \cdot)\right\rangle -\int_{0}^{T}dt \left\langle  \rho(t, \cdot), \partial_{t}G(t, \cdot)\right\rangle \\
 & -\frac{1}{2}\int_{0}^{T}dt \left\langle \Phi( \rho(t, \cdot)), \Delta G(t, \cdot)+[\nabla G(t, \cdot)]^{2}\right\rangle  .
\end{align*}
In the previous formula and in what follows, we use $\lan\cdot, \cdot\ran$ for the $L^2$-inner product
of functions $f_1,f_2:\T\to\R$ or the integral of a function $f:\T\to\R$ with respect to a measure $m$ on $\T$ whenever these expressions are well defined:
\begin{align*}
\lan f, g \ran :=\int_\T f(\theta)g(\theta)d\theta \quad \text{or} \quad \lan m, f\ran:=\int_\T f(\theta)m(d\theta).
\end{align*}

For each $\pi\in\Omega_T^\ac$ with $\pi_{t}(d\theta)= \rho(t, \theta)d\theta$ for $t\in[0, T]$,
we define the energy $\mc E(\pi)$ by
\begin{align*}
\mc E(\pi):=\sup_{G\in C^{0,1}([0,T]\times\T)}
\left\{ \int_0^Tdt  \left\lan \Phi( \rho(t,\cdot)), 2\nabla G(t,\cdot) -\left[ G(t,\cdot) \right]^2 \right\ran  \right\}.
\end{align*}
For $\pi\in\Omega_T\setminus\Omega_T^\ac$, we set $\mc E(\pi):=\infty$.
It is standard to see that $\mc E(\pi)$ is finite if and only if there exists a generalized derivative of $\Phi( \rho(t,\theta))$,
denoted by $\nabla\Phi( \rho(t,\theta))$, which belongs to $L^1([0,T]\times\T)$ and satisfies
\begin{align*}
\mc E(\pi)=\int_0^Tdt\int_\T d\theta \dfrac{|\nabla \Phi( \rho(t,\theta))|^2}{\Phi( \rho(t,\theta))}<\infty.
\end{align*}
See \cite[Chapter 5, Section 7]{kl99} for the details.
Moreover, under assumptions \eqref{LG} and \eqref{M},
$\mc E(\pi)<\infty$ implies that 
$ \rho(t,\theta)$ has a generalized derivative, denoted by $\nabla \rho(t,\theta)$, which belongs to $L^1([0,T]\times\T)$ and satisfies
\begin{align*}
\int_0^Tdt\int_\T d\theta \dfrac{|\nabla  \rho(t,\theta)|^2}{ \rho(t,\theta)} \le C\mc E(\pi)<\infty,
\end{align*}
for some constant $C>0$. Indeed, the first inequality follows from a chain rule and the bounds \eqref{boundPhi} and \eqref{boundPhi'}.

Now we explain the rate function corresponding to the large deviation principle for the empirical measure.
The dynamical rate function $I_{\textup{dyn}}:\Omega_{T}\to[0, \infty]$ is defined by
\begin{align*}
I_{\textup{dyn}}(\pi):=
\begin{cases}
\displaystyle\sup_{G\in C^{1, 2}([0, T]\times\mathbb{T})} \mc J_{G}(\pi), \quad \text{if $\pi\in\Omega_T^\ac$ and $\mc E(\pi)<\infty$},\\
\infty, \quad \text{otherwise}.
\end{cases}
\end{align*}
We next define the initial rate function $I_{\textup{\ensuremath{\alpha}}}^{0}$
with respect to the measure $\nu_{\alpha}^N$. If a measure $m\in\mc {M}_{+}(\mathbb{T})$
is not absolutely continuous with respect to the Lebesgue measure on $\T$,
then we define $I_{\textup{\ensuremath{\alpha}}}^{0}(m):=\infty$.
Otherwise, we write $m(d\theta)=\rho(\theta)d\theta$ and define
\begin{equation*}
I_{\textup{\ensuremath{\alpha}}}^{0}(m):=
\int_{\mathbb{T}}\left[\rho(\theta)\log\frac{\Phi(\rho(\theta))}{\Phi(\alpha)}-\log\frac{Z(\Phi(\rho(\theta)))}{Z(\Phi(\alpha))}\right]d\theta,
\end{equation*}
with the usual convention $0\log 0:= 0$.
The full large deviation rate function is now defined by
\[
I_{\alpha}(\pi):=I_{\textup{dyn}}(\pi)+I_{\textup{\ensuremath{\alpha}}}^{0}(\pi_{0}) .
\]

We recall the following large deviation principle for the empirical measure.

\begin{theorem}The sequence of measures $\{\bb Q_{\alpha}^{N}\}_{N=1}^{\infty}$
satisfies a large deviation principle with good rate function $I_{\alpha}$ and speed $N$.
Namely, for every closed set $\mf C$ of $\Omega_T$,
\begin{align*}
\varlimsup_{N\to\infty}\dfrac1N\log \bb Q_\alpha^N\left(\mf C \right) \le - \inf_{\pi\in\mf C} I_\alpha(\pi),
\end{align*}
and for every open set $\mf O$ of $\Omega_T$,
\begin{align*}
\varliminf_{N\to\infty}\dfrac1N\log \bb Q_\alpha^N\left(\mf O \right) \ge - \inf_{\pi\in\mf O} I_\alpha(\pi).
\end{align*}
\end{theorem}

The above large deviation principle follows by combining the upper and lower bound arguments of \cite{bkl95}
 with the identification of the corresponding rate functions developed in \cite{fg23,fgh25}.
 The probabilistic estimates in \cite[Theorem 3.1]{fgh25} apply under \eqref{LG} and \eqref{M},
 while the argument of \cite[Lemma 4.2]{fgh25} justifies the finite-energy restriction in the definition of $I_{\mathrm{dyn}}$.
 To extend the lower bound from regular trajectories to all trajectories with finite rate,
 we use the approximation argument of \cite[Theorem 8.3]{fgh25}.

The hypotheses of this approximation result hold because $\Phi$ is smooth and satisfies \eqref{boundPhi} and \eqref{boundPhi'},
which imply Assumptions 5.1, 6.3 and 7.1 of \cite{fgh25}, with $\sigma:=\sqrt{\Phi}$ in Assumption 5.1.
Although \cite{fgh25} is formulated on $\mathbb R^d$,
the relevant proofs adapt to the torus by using periodic mollifiers and omitting spatial cutoffs at infinity.
Consequently, the upper and lower rate functions coincide with $I_\alpha$ defined above.
In particular, no additional monotonicity assumption on $g$ is required.

\subsection{Large deviations for additive functionals}

We denote by $\bb P_\alpha^N$ the law of the process and by $\bb E_\alpha^N$ expectation with respect to $\bb P_\alpha^N$. 

Fix a local Lipschitz function $h:\N^\Z\to\R$. We denote by $\Gamma_h^N$ the additive functional defined by
\begin{align*}
\Gamma_h^N :=  \int_0^T h(\eta_t^N) dt .
\end{align*}
We also define
\begin{align*}
H(\rho):=E_{\nu_\rho}[h], \quad \rho\in\R_+.
\end{align*}
It is well known that the function $H$ is Lipschitz continuous.

For $0<\e<1/2$, let $\iota_\e:\T\to[0,\infty)$ be the function
$\iota_\e(\theta):=(2\e)^{-1}{\bf 1}_{[-\e,\e]}(\theta)$.
Then we define the functions $\psi_\e, \psi$ on $\R$ by
\begin{align}
\psi_\e(a) &:=  \inf \left\{ I_\alpha(\pi): \pi\in\Omega_T, \int_0^T H\left(\lan \pi_t, \iota_\e \ran \right) dt = a \  \right\}  , \notag \\
\psi(a) &:=  \lim_{\delta\to0} \varliminf_{\e\to0} \inf_{b: |b - a|\le\delta} \psi_\e(b), \label{2-3}
\end{align}
respectively.

The following theorem states the large deviation principle for $\Gamma_h^N$.

\begin{theorem}\label{thm5-3}
The sequence $\{\Gamma_h^N\}_{N\in\N}$ satisfies a large deviation principle with good rate function $\psi$
and speed $N$. Namely, for every closed set $C$ of $\R$,
\begin{align*}
\varlimsup_{N\to\infty} \dfrac1N \log \bb P_\alpha^N\left(\Gamma_h^N\in C \right) \le -\inf_{a\in C}\psi(a),
\end{align*}
and for every open set $O$ of $\R$,
\begin{align*}
\varliminf_{N\to\infty} \dfrac1N \log \bb P_\alpha^N\left(\Gamma_h^N\in O \right) \ge -\inf_{a\in O}\psi(a).
\end{align*}
\end{theorem}

To prove Theorem \ref{thm5-3}, we apply the contraction principle to the large deviation principle for the empirical measure.
For this purpose, we need to approximate $\Gamma_h^N$ by an appropriate function of the empirical measure.
As a first idea, one may approximate $\Gamma_h^N$ by the additive functional
\begin{align*}
\int_0^T H\left(\dfrac{1}{2\e N+1}\sum_{x:|x|\le\e N} \eta_t(x)\right) dt,
\end{align*}
which is a function of the empirical measure.
If the preceding expression is an exponentially good approximation of $\{\Gamma_h^N\}_{N\in\N}$,
\cite[Theorem 4.2.16]{dz98} yields a large deviation principle for
$\{\Gamma_h^N\}_{N\in\N}$ with rate function given by \eqref{2-3}.
In the case of the zero-range process, this procedure may fail for an unbounded local function $h$ and
we need another approximation for $\{\Gamma_h^N\}_{N\in\N}$.

For each $K\in\N$, set
\begin{align*}
h_K(\eta) := 
\begin{cases}
h(\eta) , &  \text{if $\eta(x) \le K$ for any $x\in\supp h$} ,\\
0 , & \text{otherwise} ,
\end{cases}
\end{align*}
where $\supp h$ denotes the support of $h$.
We define 
\begin{align*}
r_K^h(\eta) := h(\eta)-h_K(\eta) \quad\Longleftrightarrow\quad h(\eta) = h_K(\eta)+r_K^h(\eta) .
\end{align*}
For each $\ell,K\in\N$, we also define
$$
H_K(\rho) := E_{\nu_\rho}[h_K] , \quad H_{\ell,K}(\eta) :=  H_K(\eta^\ell(0)) , \quad 
\overline{H}_{\ell,K}(\rho):=E_{\nu_\rho}[H_{\ell,K}].
$$
Here, for $x\in\T_N$, $\eta^\ell(x)$ denotes the average over the box with center $x$ and side length $2\ell+1$:
\begin{align*}
\eta^\ell(x):=\dfrac{1}{2\ell+1}\sum_{y:|y-x|\le\ell} \eta(y).
\end{align*}
When $h(\eta)$ is given by $g(\eta(0))$,
we shall denote $h_K(\eta), H_K, H_{\ell,K}, \overline H_{\ell,K}$
by $g_K(\eta(0)),$ $\Phi_K,$ $\Phi_{\ell,K},$ $\overline \Phi_{\ell,K}$, respectively.

We recall the notion of Lipschitz functions.
A local function $h:\N^\Z\to\R$, supported on $A\subset\Z$, is said to be
Lipschitz if there exists a finite constant $C_0=C_0(h)$ such that
\begin{align*}
| h(\eta) - h(\xi) | \le  C_0\sum_{x\in A} |\eta(x) - \xi(x)| ,
\end{align*}
for any configurations $\eta, \xi\in \N^\Z$.

The following result plays a key role in proving Theorem \ref{thm5-3}.

\begin{theorem}\label{thm5-2}
Let $h$ be a local Lipschitz function on $\N^\Z$.
Then, for any $0\le t \le T$ and  $\delta>0$, we have
\begin{align*}
\varlimsup_{K\to\infty}\varlimsup_{\ell\to\infty}\varlimsup_{\e\to0}\varlimsup_{\varepsilon\to0}\varlimsup_{N\to\infty}
\dfrac{1}{N}\log \bb P_\alpha^N\left( \left| \int_0^t V(\eta_s)ds \right| \ge\delta\right)  = -\infty ,
\end{align*}
where
\begin{align*}
V(\eta) \equiv  V_{N,\varepsilon,\e,\ell,K}(\eta) :=  h(\eta) -\dfrac{1}{\e N}\sum_{x=1}^{\e N} \overline H_{\ell,K}(\eta^{\varepsilon N}(x)) .
\end{align*}
\end{theorem}

We now outline the strategy for proving Theorem \ref{thm5-3}.
By the super-exponential estimate established in Theorem \ref{thm5-2}, the additive functional $\Gamma_h^N$ admits an exponentially good approximation given by
\begin{align}\label{5-23}
\int_0^T \dfrac{1}{\e N}\sum_{x=1}^{\e N} \overline H_{\ell,K}(\eta_t^{\varepsilon N}(x))dt,
\end{align}
which is a function of the empirical measure.
Therefore, by the large deviation principle for the empirical measure and the contraction principle,
for fixed parameters $\varepsilon,\e,\ell,K$,
\eqref{5-23} satisfies the large deviation principle
with rate function
\begin{align}\label{5-27}
\psi_{\varepsilon, \e, \ell, K}(a):=
\inf \left\{ I_\alpha(\pi): \int_0^T dt \int_\T d\theta \kappa_\e(\theta)
\overline H_{\ell,K}\left(\int_\T \iota_\varepsilon(\theta_1-\theta)\rho(t,\theta_1)d\theta_1\right)=a \right\},
\end{align}
where $\kappa_\e(\theta):=\e^{-1}{\bf 1}_{[0,\e]}(\theta)$ for $\e>0$.
Letting $\varepsilon\to0, \e\to0, \ell\to\infty, K\to\infty$,
we recover the original rate function $\psi$,
yielding the additive-functional large deviation principle.
This approximation procedure, together with the spatial regularity of finite-rate trajectories, also yields the simple variational formula for $\psi$ stated in Theorem \ref{thm5-1}.
The precise proof is given in Section \ref{sec:proof}.

The occupation time at the origin is denoted by $\Gamma^N:=\Gamma_h^N$ with $h(\eta):=\eta(0)$.
For this choice of $h$, we have $H(\rho)=\rho$, and Theorem \ref{thm5-3} gives the following result.

\begin{corollary}\label{thm5-11}
The sequence $\{\Gamma^N\}_{N\in\N}$ satisfies a large deviation principle with good rate function $\psi$ and speed $N$,
where $\psi$ is defined by \eqref{2-3} with $H(\rho)=\rho$. Namely, for every closed set $C$ of $\R_+$,
\begin{align*}
\varlimsup_{N\to\infty} \dfrac1N \log \bb P_\alpha^N\left(\Gamma^N\in C \right) \le -\inf_{a\in C}\psi(a),
\end{align*}
and for every open set $O$ of $\R_+$,
\begin{align*}
\varliminf_{N\to\infty} \dfrac1N \log \bb P_\alpha^N\left(\Gamma^N\in O \right) \ge -\inf_{a\in O}\psi(a).
\end{align*}
\end{corollary}

\section{The rate function $\psi$}\label{sec:ratefunction}

In this section, we study several properties of the rate function $\psi$.
Throughout this section, $h$ denotes the local Lipschitz function fixed in Section 2.3 and $H(\rho):=E_{\nu_\rho}[h]$.

Let $\Omega_T^{I_\alpha}$ be the set of all trajectories $\pi(t,d\theta)=\rho(t,\theta)d\theta$ in $\Omega_T^\ac$ satisfying $I_\alpha(\pi)<\infty$.
For any trajectory $\pi(t,d\theta)=\rho(t,\theta)d\theta$ in $\Omega_T^{I_\alpha}$,
note that $\rho(t,0)$ is well-defined for almost all $t\in[0,T]$.
Indeed, since $\mc E(\pi)$ is finite, $\rho(t,\cdot)$ has a continuous version on $\T$ for almost all $t\in[0,T]$.
We shall use this fact without any further mention in the following argument.

The main result of this section is a variational formula for $\psi$. In what follows,
we always assume that the rate function $I_\alpha(\pi)$ is finite and therefore
the density at the origin $\rho(t,0)$ is well-defined for almost all $t\in[0,T]$.

\begin{theorem}\label{thm5-1}
The rate function $\psi$ defined in \eqref{2-3} can be expressed by the variational formula
\begin{align}\label{5-19}
\psi(a) =  \inf \left\{ I_\alpha(\pi): \pi\in \Omega_T^{I_\alpha},\ \int_0^T H(\rho(t,0)) dt = a\  \right\}, \quad a\in\R.
\end{align}
Moreover, if $\psi(a)<\infty$, the infimum is achieved by some trajectory $\pi\in \Omega_T^{I_\alpha}$
satisfying $\int_0^T H(\rho(t,0)) dt=a$.
\end{theorem}

Before proving Theorem \ref{thm5-1}, we shall prepare several auxiliary lemmata.
We start by providing a fundamental bound for the energy $\mc E(\pi)$.

\begin{proposition}\label{energyest}
There exists a constant $C_0>0$ such that for any $\pi\in\Omega_T$,
\begin{align*}
\mc E(\pi)\le C_0(I_\alpha(\pi)+1).
\end{align*}
\end{proposition}

\begin{proof}
It suffices to consider $\pi$ with $I_\alpha(\pi)<\infty$.
The verification of the hypotheses on $\Phi$ and the adaptation
to the periodic setting are as discussed after Theorem 2.1.
Combining the control representation of the dynamical rate function,
as in \cite[Proposition 37]{fg23}, with the entropy dissipation estimate
and the equivalence of weak and renormalized kinetic solutions in
\cite[Proposition 6.4 and Theorem 7.3]{fgh25}, we obtain
\begin{align*}
\int_0^T dt\int_{\mathbb T} d\theta
\left|\nabla\sqrt{\Phi(\rho(t,\theta))}\right|^2
\le C\bigl(I_\alpha(\pi)+1\bigr).
\end{align*}
Here the relative entropy with reference density $\alpha$ coincides
with $I_\alpha^0$, and the factor $1/2$ in our diffusion coefficient
only changes the constant.
The conclusion follows from the fact that
\begin{align*}
\mathcal E(\pi)
=4\int_0^T dt\int_{\mathbb T} d\theta
\left|\nabla\sqrt{\Phi(\rho(t,\theta))}\right|^2.
\end{align*}

\end{proof}

We next complement the energy estimate with a bound on the total mass in terms of the rate function.

\begin{proposition}\label{prop5-7}
There exists a constant $C_0>0$ such that for any $m\in\mc M_+(\T)$,
\begin{align*}
m (\T) \le C_0(I_\alpha^0(m)+1).
\end{align*}
In particular, there exists a constant $C_0>0$ such that for any $\pi\in\Omega_T$,
\begin{align*}
\pi_0(\T) + \mc E(\pi)\le C_0(I_\alpha(\pi)+1).
\end{align*}
\end{proposition}
\begin{proof}
Since $I_\alpha^0$ is convex, for any $m\in\mc M_+(\T)$ with $m(\T)<\infty$, we have
\begin{align}\label{5-28}
I_\alpha^0(m) &\ge  m(\T)\log\left( \dfrac{\Phi(m(\T))}{\Phi(\alpha)}\right) - \log\left( \dfrac{Z(\Phi(m(\T)))}{Z(\Phi(\alpha))}\right) \notag\\
 &\ge c_0m(\T)\log m(\T) - C_0\left( m(\T) + 1 \right) ,
\end{align}
for some constants $c_0, C_0>0$, which are independent of $m$.
This shows the first part of the proposition.
The second part follows from the first part, Proposition \ref{energyest} and the definition of $I_\alpha$.
\end{proof}

Combining the preceding energy and mass estimates,
we obtain an $L^2$-bound for trajectories with finite rate functions.

\begin{lemma}\label{lem5-8}
There exists a constant $C_0>0$ such that 
for any trajectory $\pi(t,d\theta)=\rho(t,\theta)d\theta$
in $\Omega_T^{I_\alpha}$, we have
\begin{align*}
\int_0^T \int_{\T} \rho(t,\theta)^2 dt d\theta \le
C_0\left( I_\alpha(\pi)^2+1 \right).
\end{align*}
\end{lemma}
\begin{proof}
Fix a trajectory $\pi(t,d\theta)=\rho(t,\theta)d\theta$ in $\Omega_T^{I_\alpha}$.
We begin with
\begin{align*}
\left| \rho(t,\theta)-\rho(t,\theta_1) \right|&\le\int_{\T}|\nabla\rho(t,\theta_2)|d\theta_2 \\
&\le \left( \int_{\T} \rho(t,\theta_2)d\theta_2\right)^{1/2}\left( \int_{\T} \dfrac{(\nabla\rho(t,\theta_2))^2}{\rho(t,\theta_2)}d\theta_2\right)^{1/2}\\
&= \pi_0(\T)^{1/2}\left( \int_{\T} \dfrac{(\nabla\rho(t,\theta_2))^2}{\rho(t,\theta_2)}d\theta_2\right)^{1/2},
\end{align*}
for any $\theta, \theta_1\in\T$.
Then integrating in $\theta_1$, we have
\begin{align}\label{5-21}
\rho(t,\theta)\le \pi_0(\T)^{1/2}\left( \int_{\T} \dfrac{(\nabla\rho(t,\theta_1))^2}{\rho(t,\theta_1)}d\theta_1\right)^{1/2}
+ \pi_0(\T).
\end{align}
Using the elementary inequality $2ab\le A^{-1}a^2+Ab^2$ and integrating \eqref{5-21} in $t$ and $\theta$, we have
\begin{align*}
\int_0^T \int_{\T} \rho(t,\theta)^2 dt d\theta
\le 2 \pi_0(\T) \left( \int_0^Tdt \int_{\T} \dfrac{(\nabla\rho(t,\theta))^2}{\rho(t,\theta)}d\theta\right) + 2T \pi_0(\T)^2.
\end{align*}
Then the conclusion follows from Propositions \ref{energyest} and \ref{prop5-7}.
\end{proof}

We now use the preceding bounds to control the error
in approximating the density at a fixed point by spatial averages.

\begin{lemma}\label{lem5-6}
Fix a Lipschitz continuous function $F:\R_+\to\R$.
There exists a constant $C_0>0$ such that,
for any $0<\e<1/2$ and trajectory $\pi(t,d\theta)=\rho(t,\theta)d\theta$ in $\Omega_T^{I_\alpha}$,
\begin{align*}
\int_0^T  \left| F\left( \rho(t,0) \right) -  F\left( \lan \pi_t, \iota_\e \ran \right) \right| dt
 \le  C_0\e^{1/4}\left( I_\alpha(\pi) +1\right).
\end{align*}
\end{lemma}
\begin{proof}
The left-hand side of the statement can be bounded as
\begin{align*}
\|F\|_\Lip \int_0^T dt \left| \int_\T d\theta \iota_\e(\theta)\left[ \rho(t,0) - \rho(t,\theta)\right] \right|
\le \|F\|_\Lip \int_0^T dt \int_\T d\theta  \iota_\e(\theta)\int_{-|\theta|}^{|\theta|} \left| \nabla\rho(t,\theta_1) \right| d\theta_1,
\end{align*}
where $\|F\|_\Lip$ denotes the Lipschitz constant of $F$.
By the elementary inequality $2ab\le A^{-1}a^2+Ab^2$ and the Schwarz inequality, we have
\begin{align*}
\int_{-|\theta|}^{|\theta|} \left| \nabla\rho(t,\theta_1) \right| d\theta_1
&\le \dfrac{1}{2\e^{1/4}}\int_{-|\theta|}^{|\theta|}\rho(t,\theta_1)d\theta_1 +\dfrac{\e^{1/4}}{2} \int_{\T} \dfrac{(\nabla\rho(t,\theta_1))^2}{\rho(t,\theta_1)}d\theta_1.
\end{align*}
Note that $\iota_\e(\theta)$ vanishes for $|\theta|>\e$.
Therefore we have
\begin{align*}
&\int_0^T \left| F\left(\rho(t,0)\right) - F\left(\lan \pi_t, \iota_\e \ran \right)  \right| dt\\
&\quad\le \dfrac{\|F\|_\Lip}{2\e^{1/4}} \int_0^T \int_{-\e}^\e \rho(t,\theta) dtd\theta
+\dfrac{\|F\|_\Lip\e^{1/4}}{2}\int_0^T \int_{\T} \dfrac{(\nabla\rho(t,\theta))^2}{\rho(t,\theta)}dtd\theta\\
&\quad\le \dfrac{\|F\|_\Lip T^{1/2}\e^{1/4}}{2^{1/2}} \left( \int_0^T \int_{\T} \rho(t,\theta)^2 dt d\theta \right)^{1/2}
+\dfrac{\|F\|_\Lip\e^{1/4}}{2}\int_0^T \int_{\T} \dfrac{(\nabla\rho(t,\theta))^2}{\rho(t,\theta)}dtd\theta,
\end{align*}
where we used the Schwarz inequality to obtain the second inequality.
Then the conclusion follows from Proposition \ref{energyest} and Lemma \ref{lem5-8}.
\end{proof}

As a consequence of the previous lemma,
in the following lemma we show that,
for each Lipschitz continuous function $F:\mathbb{R}_+\to\mathbb{R}$, the map
$\pi\mapsto \int_0^T F(\rho(t,0)) dt$ is continuous
along convergent sequences of trajectories with uniformly bounded rate.

\begin{lemma}\label{lem5-7}
Fix a Lipschitz continuous function $F:\R_+\to\R$.
Let $\{\pi^n(t,d\theta)=\rho^n(t,\theta)d\theta\}_{n\in\N}$ be a sequence of trajectories in $\Omega_T^{I_\alpha}$
converging to some $\pi(t,d\theta)=\rho(t,\theta)d\theta$ in $\Omega_T^{I_\alpha}$ as $n\to\infty$.
Assume that the sequence $\{ I_\alpha(\pi^n)\}_{n\in\N}$ is bounded. Then,
\begin{align*}
\lim_{n\to\infty} \int_0^T F\left( \rho^n(t,0) \right)dt  =  \int_0^T F\left( \rho(t,0) \right) dt .
\end{align*}
\end{lemma}
\begin{proof}
Fix $\e>0$. Since the sequence $\{ I_\alpha(\pi^n)\}_{n\in\N}$ is bounded, Lemma \ref{lem5-6} gives
\begin{align*}
\left| \int_0^T F\left( \rho^n(t,0) \right) dt - \int_0^T F\left( \rho(t,0) \right) dt \right|
 \le  \left| \int_0^T F\left( \lan\pi_t^n, \iota_\e \ran \right) dt - \int_0^T F\left(  \lan\pi_t, \iota_\e \ran \right) dt \right|
+ C_0\e^{1/4} ,
\end{align*}
for some constant $C_0>0$.
For each $\e>0$, since $\pi^n$ converges to $\pi$ in $\Omega_T$,
the first term on the right-hand side of the preceding inequality vanishes as $n\to\infty$.
Therefore, by first letting $n\to\infty$ and then $\e\to0$, we obtain the conclusion.
\end{proof}

We are now in the position to prove Theorem \ref{thm5-1}.

\begin{proof}[Proof of Theorem \ref{thm5-1}]
Let
\begin{align*}
\widehat\psi(a):=\inf\left\{I_\alpha(\pi):\pi\in\Omega_T^{I_\alpha},\ \int_0^T H(\rho(t,0))dt=a\right\}.
\end{align*}
Fix a trajectory $\pi(t,d\theta)=\rho(t,\theta)d\theta\in\Omega_T^{I_\alpha}$ satisfying
$\int_0^T H(\rho(t,0))dt=a$.
By Lemma \ref{lem5-6} with $F=H$,
\begin{align*}
\lim_{\e\to0}\int_0^T H\left(\lan\pi_t,\iota_\e\ran\right)dt=a.
\end{align*}
Therefore, by the definition \eqref{2-3}, $\psi(a)\le I_\alpha(\pi)$, and hence $\psi(a)\le\widehat\psi(a)$.

Suppose now that $\psi(a)<\infty$.
By the definition \eqref{2-3}, there exist sequences $\delta_n\downarrow0$, $\e_n\downarrow0$, $a_n\to a$ and
$\pi^n(t,d\theta)=\rho^n(t,\theta)d\theta\in\Omega_T^{I_\alpha}$ such that
\begin{align*}
\int_0^T H\left(\lan\pi_t^n,\iota_{\e_n}\ran\right)dt=a_n,
\quad
\lim_{n\to\infty} I_\alpha(\pi^n)=\psi(a).
\end{align*}
Since $I_\alpha$ is good, taking a subsequence if necessary, we can find a trajectory
$\pi(t,d\theta)=\rho(t,\theta)d\theta\in\Omega_T^{I_\alpha}$ such that $\pi^n$ converges to $\pi$ in $\Omega_T$ and
\begin{align*}
I_\alpha(\pi)\le\varliminf_{n\to\infty}I_\alpha(\pi^n)=\psi(a).
\end{align*}
By Lemmata \ref{lem5-6} and \ref{lem5-7} with $F=H$,
\begin{align*}
\lim_{n\to\infty}\int_0^T H(\rho^n(t,0))dt
=\int_0^T H(\rho(t,0))dt,
\end{align*}
while
\begin{align*}
\left|\int_0^T H(\rho^n(t,0))dt-a_n\right|
\le C_0\e_n^{1/4}(I_\alpha(\pi^n)+1)\longrightarrow0.
\end{align*}
Thus $\int_0^T H(\rho(t,0))dt=a$, and therefore $\widehat\psi(a)\le I_\alpha(\pi)\le\psi(a)$.
This proves \eqref{5-19} when $\psi(a)<\infty$ and also proves the second assertion.
If $\psi(a)=\infty$, the inequality $\psi(a)\le\widehat\psi(a)$ proved above implies $\widehat\psi(a)=\infty$.
\end{proof}

As a consequence of Theorem \ref{thm5-1},
we obtain several properties of the rate function for the occupation time.

\begin{corollary}\label{cor5-1}
Let $h(\eta):=\eta(0)$, so that $H(\rho)=\rho$.
Then $\psi(a)<\infty$ for every $a\in\R_+$, and $\psi(a)$ vanishes if and only if $a=\alpha T$.
If, in addition, $g(k)=k$ for every $k\in\N$, then the rate function $\psi$ is convex and continuous on $\R_+$.
\end{corollary}
\begin{proof}
Fix $a\in\R_+$ and consider the constant trajectory
$\pi_t^{(a)}(d\theta):=(a/T)d\theta$, $t\in[0,T]$.
It satisfies $\int_0^T\rho^{(a)}(t,0)dt=a$.
Since this trajectory solves the hydrodynamic equation, $I_{\mathrm{dyn}}(\pi^{(a)})=0$, while
$I_\alpha^0(\pi_0^{(a)})<\infty$.
Therefore, by the variational formula \eqref{5-19},
$\psi(a)\le I_\alpha(\pi^{(a)})<\infty$.

It is clear that $\psi(\alpha T)=0$.
Conversely, assume that $\psi(a)=0$ and $a\neq\alpha T$.
Then, by Theorem \ref{thm5-1}, there exists a trajectory
$\pi(t,d\theta)=\rho(t,\theta)d\theta$ in $\Omega_T^{I_\alpha}$ such that
$\psi(a)=I_\alpha(\pi)$ and $\int_0^T \rho(t,0)dt = a$.
Since $I_\alpha(\pi)$ vanishes, $\pi_t(d\theta)$ is equal to $\alpha d\theta$ for any $t\in[0,T]$.
This contradicts the fact that $\int_0^T \rho(t,0)dt=a$ and $a\neq\alpha T$.

Suppose now that $g(k)=k$ for every $k\in\N$.
Then $\Phi(\rho)=\rho$ and $Z(\phi)=e^\phi$.
We first observe that $I_\alpha$ is convex.
Indeed, in this case $\mc E$ is convex, since it is the supremum of linear functionals of $\rho$,
and hence its finite-energy domain is convex.
Moreover, for each $G\in C^{1,2}([0,T]\times\T)$, the functional $\mc J_G$ is affine in $\rho$.
It follows that $I_{\mathrm{dyn}}$ is convex.
The initial rate function is given by
\begin{align*}
I_\alpha^0(m)
=\int_\T \left\{\rho(\theta)\log\frac{\rho(\theta)}{\alpha}-\rho(\theta)+\alpha\right\}d\theta,
\end{align*}
which is also convex.
Therefore, $I_\alpha$ is convex.
Since the constraint $\int_0^T\rho(t,0)dt=a$ in \eqref{5-19} is linear in $\rho$,
the variational formula implies that $\psi$ is convex on $\R_+$.

We have already shown that $\psi$ is finite on $\R_+$.
Hence, as a finite convex function, $\psi$ is continuous on $(0,\infty)$.
It remains to check continuity at $0$.
Since $\psi$ is a good rate function, it is lower semicontinuous.
Fix $b>0$. For every $0<a<b$, convexity gives
\begin{align*}
\psi(a)
\le \left(1-\frac{a}{b}\right)\psi(0)+\frac{a}{b}\psi(b).
\end{align*}
Consequently,
$\varlimsup_{a\downarrow0}\psi(a)\le\psi(0)$,
while lower semicontinuity yields
$\psi(0)\le\varliminf_{a\downarrow0}\psi(a)$.
Hence $\psi$ is continuous at $0$, which completes the proof.
\end{proof}

\begin{remark}\label{rem5-1}
Under assumptions (LG) and (M), the convexity of $I_\alpha$ used above is not available in general,
so the preceding convexity argument does not apply.
Nevertheless, the continuity of $\psi$ can also be proved under these assumptions
by adapting the approximation arguments of \cite{fg23,fgh25}.
To avoid developing these arguments and to keep the presentation simple, we restrict the continuity statement above to the case $g(k)=k$.
\end{remark}

Recall the definition of the function $\psi_{\varepsilon, \e, \ell, K}$ from \eqref{5-27}.
In the rest of this section,
we show that $\psi$ can be recovered from $\psi_{\varepsilon, \e, \ell, K}$
by successively removing the approximation parameters.
We also study several miscellaneous lemmata, which are needed for the proof of Theorems \ref{thm5-3}, \ref{thm5-2}.

\begin{proposition}\label{prop5-1}
Let $h$ be a local Lipschitz function and define $H,H_K,\overline H_{\ell,K}$ as in Section 2.3.
For any $A>0$ and each fixed $K\in\N$,
\begin{align*}
\lim_{\ell\to\infty}\sup_{0\le\rho\le A}
\left|\overline H_{\ell,K}(\rho)-H_K(\rho)\right|=0,
\end{align*}
and
\begin{align*}
\lim_{K\to\infty}\sup_{0\le\rho\le A}|H_K(\rho)-H(\rho)|=0.
\end{align*}
In particular,
\begin{align*}
\varlimsup_{K\to\infty}\varlimsup_{\ell\to\infty}\sup_{0\le\rho\le A}
\left|H(\rho)-\overline H_{\ell,K}(\rho)\right|=0.
\end{align*}
\end{proposition}
\begin{proof}
For each fixed $K$, the function $H_K$ is bounded and Lipschitz. Computing the variance of $\eta^\ell(0)$ under $\nu_\rho$,
we obtain the first assertion uniformly over $0\le\rho\le A$.
On the other hand, by the Lipschitz property of $h$,
\begin{align*}
|h(\eta)-h_K(\eta)|
\le C_0\left(1+\sum_{x\in\supp h}\eta(x)\right)
{\bf 1}\left\{\max_{x\in\supp h}\eta(x)>K\right\}.
\end{align*}
The one-site marginals of $\nu_\rho$, $0\le\rho\le A$, have a uniformly bounded exponential moment.
Therefore, the right-hand side is uniformly integrable, which proves the second assertion.
The last assertion follows from the first two.
\end{proof}

We next control the contribution of large values of the density at the observed point,
uniformly at the level of the rate function.

\begin{proposition}\label{prop5-2}
For any $a>0$, we have
\begin{align}\label{5-8}
\lim_{A\to\infty}\sup_{\pi\in\Omega_T^{I_\alpha}}
\left\{ a\int_0^T \rho(t,0) {\bf 1}\{\rho(t,0)\ge A\}dt - I_\alpha(\pi) \right\} 
 =  0 .
\end{align}
\end{proposition}
\begin{proof}
The first step is to reduce the problem to estimating the supremum over trajectories with bounded mass.
Take $\pi(t,d\theta)=\rho(t,\theta)d\theta$ in $\Omega_T^{I_\alpha}$.

Recall the bound \eqref{5-21}.
Integrating \eqref{5-21} in $t$ and applying the Schwarz inequality give
\begin{align}\label{5-9}
\int_0^T \rho(t,\theta)dt
&\le  C(T\pi_0(\T)\mc E(\pi))^{1/2} + T\pi_0(\T) \notag \\
&\le (1/2\beta)\mc E(\pi) + T(1+C^2\beta/2)\pi_0(\T) \notag \\
&\le (C_0/2\beta)(1+I_\alpha(\pi)) + T(1+C^2\beta/2)\pi_0(\T),
\end{align}
for any $\theta\in\T$ and $\beta>0$, where in the third line we used Proposition \ref{energyest}.

Letting $\theta=0$ and $\beta=aC_0$ in \eqref{5-9}, by \eqref{5-28}, we obtain
\begin{align*}
a\int_0^T \rho(t,0) {\bf 1}\{\rho(t,0)\ge A\}dt - I_\alpha(\pi) \le C_1(1+\pi_0(\T)) - c_0\pi_0(\T)\log\pi_0(\T),
\end{align*}
for some constants $c_0, C_1>0$, which are independent of $\pi$ and $A$.
Therefore, the expression inside the braces in \eqref{5-8} is
nonpositive whenever $\pi_0(\T)\ge C_2$ for some constant $C_2>0$, which depends on $a$ but
does not depend on $\pi$ and $A$.
Thus, we may restrict the supremum in \eqref{5-8}
to trajectories satisfying $\pi_0(\T)\le C_2$.

In what follows, we use the elementary inequality $(a+b)^p\le2^pa^p+2^pb^p$
without any further mention. By the Schwarz inequality, \eqref{5-21}, \eqref{5-9} with $\theta=0$ and Proposition \ref{prop5-7}, we have
\begin{align}\label{5-12}
\int_0^T \rho(t,0) {\bf 1}\{\rho(t,0)\ge A\}dt
 &\le  \left(\int_0^T \rho(t,0)^2dt\right)^{1/2}\left(\int_0^T {\bf 1}\{\rho(t,0)\ge A\}dt\right)^{1/2}\notag\\
 &\le \dfrac{C_2}{A^{1/2}}(1+\mc E(\pi)^{1/2})\left(\int_0^T \rho(t,0)dt\right)^{1/2}\notag\\
 &\le \dfrac{C_3}{A^{1/2}}(1+I_\alpha(\pi)^{1/2})(1+\mc E(\pi)^{1/4}) \notag\\
 &\le \dfrac{C_4}{A^{1/2}}(1+I_\alpha(\pi)^{3/4}).
\end{align}
The conclusion of the proposition easily follows from \eqref{5-12}.
\end{proof}

\begin{remark}\label{rem5-2}
In the paragraph following \eqref{5-9}, we have in fact proved the following:
for any $\gamma>0$, there exists a constant $C_\gamma>0$ such that
for any $\theta\in\T$ and $\pi\in\Omega_T^{I_\alpha}$,
\begin{align}\label{5-29}
\int_0^T \rho(t,\theta)dt \le \gamma I_\alpha(\pi)+ C_\gamma.
\end{align}
In particular, by the Lipschitz continuity of $H$,
Theorem \ref{thm5-1} and taking $\theta=0$ in \eqref{5-29}, we have
$\lim_{|a|\to\infty}\psi(a)=\infty$.
\end{remark}

Recall the definition of $\kappa_\e$.
For each $\pi\in\Omega_T$, we define
\begin{align*}
J_{\varepsilon,\e,\ell,K}(\pi):=  \int_0^T dt \int_\T d\theta \kappa_\e(\theta)\left[ \Phi\left(\left\langle \pi_t,\iota_\varepsilon(\cdot-\theta)\right\rangle\right)
- \overline\Phi_{\ell,K}\left(\left\langle \pi_t,\iota_\varepsilon(\cdot-\theta)\right\rangle\right)  \right].
\end{align*}
For each $\pi\in\Omega_T^{I_\alpha}$, we also define
\begin{align*}
J_{\ell,K}(\pi) := \int_0^T \left[ \Phi(\rho(t,0)) - \overline\Phi_{\ell,K}(\rho(t,0))  \right] dt.
\end{align*}

The following proposition allows us to remove the two spatial regularizations
in the variational expression.

\begin{proposition}\label{prop5-6}
For any $a>0, \ell,K\in\N$,
\begin{align*}
\varlimsup_{\e\to0}\varlimsup_{\varepsilon\to0}
\sup_{\pi\in \Omega_T}\left\{ aJ_{\varepsilon,\e,\ell,K}(\pi) - I_\alpha(\pi) \right\} 
=  \sup_{\pi \in \Omega_T^{I_\alpha}}\left\{ aJ_{\ell,K}(\pi) - I_\alpha(\pi) \right\} .
\end{align*}
\end{proposition}

\begin{proof}
Fix $a>0, \ell,K\in\N$ and set $R_{\ell,K}(\rho):=\Phi(\rho)-\overline\Phi_{\ell,K}(\rho)$, $\rho\in\R_+$.
Since $R_{\ell,K}$ is Lipschitz (with a Lipschitz constant that may depend on $\ell$ and $K$) and $R_{\ell,K}(0)=0$,
there exists a constant $C_0>0$ such that $|R_{\ell,K}(\rho)|\le C_0\rho$, for any $\rho\in\R_+$.
Taking $\gamma=(2aC_0)^{-1}$ in \eqref{5-29} shows that, for any $\pi\in\Omega_T$
\begin{align*}
aJ_{\varepsilon,\e,\ell,K}(\pi) - I_\alpha(\pi) \le -(1/2)I_\alpha(\pi) + C_1,
\end{align*}
for some constant $C_1>0$, which is independent of $\varepsilon$ and $\e$.
Therefore, we may restrict the supremum on the left-hand side of the statement to trajectories satisfying $I_\alpha(\pi)\le C_2$, for some constant $C_2>0$, which is independent of $\varepsilon$ and $\e$.

Since the rate function $I_\alpha$ is good and lower semicontinuous, for any $\varepsilon, \e>0$,
there exists a trajectory $\pi^{\varepsilon, \e}(t,d\theta)=\rho^{\varepsilon, \e}(t,\theta)d\theta\in\Omega_T^\ac$ such that
\begin{align*}
\sup_{\pi\in \Omega_T}\left\{ aJ_{\varepsilon,\e,\ell,K}(\pi) - I_\alpha(\pi) \right\} = aJ_{\varepsilon,\e,\ell,K}(\pi^{\varepsilon, \e}) - I_\alpha(\pi^{\varepsilon, \e}),
\end{align*}
and $I_\alpha(\pi^{\varepsilon, \e})\le C_2$. Moreover, taking a subsequence if necessary, we may assume that
for any fixed $\e$,  the sequence $\{\pi^{\varepsilon, \e}\}_{\varepsilon>0}$ converges to some $\pi^\e(t,d\theta)=\rho^\e(t,\theta)d\theta$ in $\Omega_T^{I_\alpha}$ with $I_\alpha(\pi^\e)\le C_2$.
We now claim that
\begin{align*}
\lim_{\varepsilon\to0}J_{\varepsilon,\e,\ell,K}(\pi^{\varepsilon,\e})=
\int_0^T dt \int_\T d\theta \kappa_\e(\theta) R_{\ell,K}\left( \rho^\e(t,\theta) \right).
\end{align*}
Indeed, let us estimate
\begin{align*}
&\left| J_{\varepsilon,\e,\ell,K}(\pi^{\varepsilon,\e}) - \int_0^T dt \int_\T d\theta \kappa_\e(\theta) R_{\ell,K}\left( \rho^\e(t,\theta) \right) \right|\\
&\quad \le \left| \int_0^T dt \int_\T d\theta \kappa_\e(\theta) \left[ R_{\ell,K}\left(\int_\T \iota_\varepsilon(\theta_1-\theta)\rho^{\varepsilon,\e}(t,\theta_1)d\theta_1\right)
- R_{\ell,K}\left( \rho^{\varepsilon,\e}(t,\theta )\right) \right] \right| \\
&\qquad + \left| \int_0^T dt \int_\T d\theta \kappa_\e(\theta) \left[ R_{\ell,K}\left( \rho^{\varepsilon,\e}(t,\theta) \right)
-  R_{\ell,K}\left( \rho^\e(t,\theta) \right) \right] \right|.
\end{align*}
By Lemma \ref{lem5-6}, the second line of the previous expression is of order $\varepsilon^{1/4}$ and therefore vanishes as $\varepsilon\to0$.
The third line also vanishes as $\varepsilon\to0$ by the bounded convergence theorem and Lemma \ref{lem5-7}.
(By spatial translation, Lemma \ref{lem5-7} remains valid with $0$ replaced by any fixed $\theta\in\mathbb T$.)
Consequently, we have shown
\begin{align*}
\varlimsup_{\varepsilon\to0}\sup_{\pi\in \Omega_T}\left\{ aJ_{\varepsilon,\e,\ell,K}(\pi) - I_\alpha(\pi) \right\}
\le a\int_0^T dt \int_\T d\theta \kappa_\e(\theta) R_{\ell,K}\left( \rho^\e(t,\theta) \right) - I_\alpha(\pi^{\e}).
\end{align*}

Taking a subsequence of the sequence $\{\pi^\e\}_{\e>0}$, we may assume that the sequence $\{\pi^\e\}_{\e>0}$ converges to some $\pi(t,d\theta)=\rho(t,\theta)d\theta$ in $\Omega_T^{I_\alpha}$ with $I_\alpha(\pi)\le C_2$.
Let us write
\begin{align*}
&\int_0^T dt \int_\T d\theta \kappa_\e(\theta) R_{\ell,K}\left( \rho^\e(t,\theta) \right) - \int_0^T R_{\ell,K}\left(\rho(t,0) \right) dt\\
&\quad= \int_0^T dt \int_\T d\theta \kappa_\e(\theta) \left[ R_{\ell,K}\left( \rho^\e(t,\theta) \right) - R_{\ell,K}\left(\rho^\e(t,0)\right)  \right]\\
&\qquad + \int_0^T dt \int_\T d\theta \kappa_\e(\theta) \left[ R_{\ell,K}\left( \rho^\e(t,0) \right) - R_{\ell,K}\left(\rho(t,0)\right)\right].
\end{align*}
The second line of the previous expression is bounded above by
\begin{align*}
&\|R_{\ell,K}\|_\Lip\int_0^Tdt \int_\T d\theta \kappa_\e(\theta) \left| \rho^\e(t,\theta) - \rho^\e(t,0) \right|\\
&\quad \le \|R_{\ell,K}\|_\Lip \int_0^T dt \int_\T d\theta  \kappa_\e(\theta)\int_0^{|\theta|} \left| \nabla\rho^\e(t,\theta_1) \right| d\theta_1.
\end{align*}
Now we are in the same position as in the first step of the proof of Lemma \ref{lem5-6}.
Following the same proof and using $I_\alpha(\pi^\e)\le C_2$, we see that the second line of the penultimate expression
is bounded above by $C_3\e^{1/4}$. On the other hand, the third line of the penultimate expression is equal to
\begin{align*}
\int_0^T dt \left[ R_{\ell,K}\left( \rho^\e(t,0) \right) - R_{\ell,K}\left(\rho(t,0)\right)\right],
\end{align*}
which vanishes as $\e\to0$ by Lemma \ref{lem5-7}.
Therefore, we have shown
\begin{align*}
\varlimsup_{\e\to0}\varlimsup_{\varepsilon\to0}\sup_{\pi\in \Omega_T}\left\{ aJ_{\varepsilon,\e,\ell,K}(\pi) - I_\alpha(\pi) \right\}
\le aJ_{\ell,K}(\pi) - I_\alpha(\pi),
\end{align*}
for some $\pi\in\Omega_T^{I_\alpha}$.

It remains to show the converse inequality.
Fix any $\pi(t,d\theta)=\rho(t,\theta)d\theta\in\Omega_T^{I_\alpha}$.
Then one can see that
\begin{align*}
\lim_{\e\to0}\lim_{\varepsilon\to0} J_{\varepsilon,\e,\ell,K}(\pi)=\lim_{\e\to0}\int_0^Tdt\int_\T d\theta\kappa_\e(\theta)R_{\ell,K}(\rho(t,\theta))=J_{\ell,K}(\pi),
\end{align*}
which completes the proof of Proposition \ref{prop5-6}.
\end{proof}

Recall the definition of $\psi_{\varepsilon,\e,\ell,K}$ in \eqref{5-27}.
We can now identify $\psi$ in \eqref{2-3} by successively removing all the approximation parameters.

\begin{lemma}\label{lem5-11}
For any $a\in\R$,
\begin{align*}
\varliminf_{\delta\to0}\varliminf_{K\to\infty}\varliminf_{\ell\to\infty}\varliminf_{\e\to0} \varliminf_{\varepsilon\to0}
\inf_{b\in (a-\delta, a+\delta)} \psi_{\varepsilon, \e, \ell, K}(b) = \psi(a).
\end{align*}
\end{lemma}

\begin{proof}
We first prove the lower bound. If the left-hand side is infinite, there is nothing to prove.
Otherwise, by taking almost minimizers and using the goodness of $I_\alpha$, we may restrict the argument to trajectories
whose rate functions are bounded by a constant independent of the approximation parameters.
For fixed $\e,\ell,K$, letting $\varepsilon\to0$ and using Lemmata \ref{lem5-6}, \ref{lem5-7},
we see that the constraint in \eqref{5-27} converges along such trajectories to
\begin{align*}
\int_0^Tdt\int_\T d\theta\kappa_\e(\theta)\overline H_{\ell,K}(\rho(t,\theta)).
\end{align*}
Letting next $\e\to0$, again by Lemmata \ref{lem5-6}, \ref{lem5-7}, we see that this converges to
\begin{align*}
\int_0^T\overline H_{\ell,K}(\rho(t,0))dt.
\end{align*}
For fixed $K$, letting $\ell\to\infty$ gives $\int_0^T H_K(\rho(t,0))dt$.
Indeed, Proposition \ref{prop5-1}, applied before sending $K\to\infty$, together with the tail estimate \eqref{5-12}
shows that the difference vanishes uniformly on bounded level sets of $I_\alpha$.
Finally, letting $K\to\infty$ and using Proposition \ref{prop5-1} and \eqref{5-12} gives
\begin{align*}
\int_0^T H(\rho(t,0))dt=a.
\end{align*}
The lower semicontinuity of $I_\alpha$ and Theorem \ref{thm5-1} therefore yield
\begin{align*}
\varliminf_{\delta\to0}\varliminf_{K\to\infty}\varliminf_{\ell\to\infty}\varliminf_{\e\to0} \varliminf_{\varepsilon\to0}
\inf_{b\in (a-\delta, a+\delta)} \psi_{\varepsilon, \e, \ell, K}(b) \ge \psi(a).
\end{align*}

For the converse inequality, take any trajectory $\pi(t,d\theta)=\rho(t,\theta)d\theta\in\Omega_T^{I_\alpha}$ satisfying
$\int_0^T H(\rho(t,0))dt=a$.
By the same successive approximation, now for the fixed trajectory $\pi$,
\begin{align*}
\varlimsup_{K\to\infty}\varlimsup_{\ell\to\infty}\varlimsup_{\e\to0}\varlimsup_{\varepsilon\to0}
\left|\int_0^Tdt\int_\T d\theta\kappa_\e(\theta)
\overline H_{\ell,K}\left(\int_\T\iota_\varepsilon(\theta_1-\theta)\rho(t,\theta_1)d\theta_1\right)-a\right|=0.
\end{align*}
Hence the left-hand side of the statement is bounded above by $I_\alpha(\pi)$.
Taking the infimum over such trajectories and using Theorem \ref{thm5-1} gives the converse inequality.
\end{proof}

Let $C$ be a closed subset of $\R$. For each $\delta>0$, let
\begin{align*}
C^\delta:=\left\{a\in\R: \inf_{b\in C}|a -b|\le \delta \right\}.
\end{align*}

\begin{lemma}\label{lem5-10}
For any closed subset $C$ of $\R$, 
\begin{align*}
\varliminf_{\delta\to0}\varliminf_{K\to\infty}\varliminf_{\ell\to\infty}\varliminf_{\e\to0} \varliminf_{\varepsilon\to0}
\inf_{a\in C^\delta}\psi_{\varepsilon, \e, \ell, K}(a) = \inf_{a\in C}\psi(a).
\end{align*}
\end{lemma}

\begin{proof}
Note that the sublevel sets of $\psi_{\varepsilon,\e,\ell,K}$ are
uniformly bounded in the approximation parameters.
Indeed, since $h$ is Lipschitz and $|h_K|\le |h|$,
there exists a constant $C_0>0$, independent of $\ell$ and $K$, such that
$|\overline H_{\ell,K}(\rho)|\le C_0(1+\rho)$ for all $\rho\ge0$.
Hence, by Remark \ref{rem5-2} and the definitions of $\iota_\varepsilon$ and $\kappa_\e$,
for every $\gamma>0$ there exists $C_\gamma>0$, independent of $\varepsilon,\e,\ell,K$, such that
$$
\left|
\int_0^Tdt\int_\T d\theta\,\kappa_\e(\theta)
\overline H_{\ell,K}\bigl(
\langle\pi_t,\iota_\varepsilon(\cdot-\theta)\rangle
\bigr)
\right|
\le \gamma I_\alpha(\pi)+C_\gamma .
$$
Therefore,
$$
\psi_{\varepsilon,\e,\ell,K}(a)
\ge \frac{|a|-C_\gamma}{\gamma}.
$$
This uniform bound reduces the proof to compact subsets $C$ of $\R$.
For compact $C$, the proof is similar to that of Lemma \ref{lem5-11} and is omitted.
\end{proof}

\section{Proof of Theorem \ref{thm5-2}}\label{sec:replacement}

In this section, we prove Theorem \ref{thm5-2}.
Throughout this section, let $h$ be a local Lipschitz function.

The proof of Theorem \ref{thm5-2} is divided into two main steps.
We first remove the high-density part of the local function in Lemma \ref{lem5-9}.
For the resulting bounded part, Lemma \ref{lem5-2} follows from
the local one-block estimate, the local two-block estimate,
and the global replacement lemma established in Propositions \ref{prop5-3}--\ref{prop5-5}.
The main additional ingredient needed for the high-density cutoff is Lemma \ref{lem5-4},
which provides a super-exponential replacement estimate for the unbounded functional $g(\eta(0))$.

One of the main difficulties is the high-density cutoff.
This step is highly nontrivial due to the lack of a spatial average.
The result is as follows.

\begin{lemma}\label{lem5-9}
For any $0\le t \le T$ and  $\delta>0$, we have
\begin{align*}
\varlimsup_{K\to\infty}\varlimsup_{N\to\infty}
\dfrac{1}{N}\log \bb P_\alpha^N\left( \left| \int_0^t \left[ h(\eta_s) - h_K(\eta_s) \right] ds \right| \ge\delta\right)  = -\infty .
\end{align*}
\end{lemma}

To prove Lemma \ref{lem5-9}, we reduce the problem to the case where $h(\eta)=g(\eta(0))$.
To see this, we observe that
\begin{align*}
\left| h(\eta)\right |  \le  |h(\mathfrak d)| + C_0 \sum_{x\in\supp h}\eta(x),
\end{align*}
where $\mathfrak{d}$ denotes the configuration defined by $\mathfrak{d}(x):=0$ for all $x$.
Indeed, letting $\eta_{\max} :=  \max\{\eta(x):x\in\supp h\}$, we have
\begin{align*}
\left| r_{K}^{h}(\eta) \right| &  =  \left| h(\eta) \right|  \mathbf{1}\left\{ \eta(x)>K\text{ for some }x\in\supp h\right\} \\
 &  \le  C_0\left[\sum_{x\in\supp h}\eta(x)\right]\mathbf{1}\left\{ \eta_{\max}>K\right\} \\
 &  \le  C_0 |\supp h| \eta_{\max}\mathbf{1}\left\{ \eta_{\max}>K\right\} \\
 &  \le  C_1g(\eta_{\max})\mathbf{1}\left\{ \eta_{\max}>K\right\} \\
 &  =  C_1\left[ g(\eta_{\max})- g_K(\eta_{\max}) \right] \\
 &  \le   C_1\sum_{x\in\supp h} \left[ g(\eta(x)) - g_K(\eta(x))\right] .
\end{align*}
Therefore, to prove Lemma \ref{lem5-9}, it is enough to show the following lemma.

\begin{lemma}\label{lem5-3}
For any $x\in\Z$, $0\le t\le T$, and $\delta>0$, we have
\begin{align*}
\varlimsup_{K\to\infty}\varlimsup_{N\to\infty}
\dfrac{1}{N}\log \bb P_\alpha^N\left( \int_0^t \left[ g(\eta_s(x)) - g_K(\eta_s(x)) \right] ds \ge\delta\right)  = -\infty .
\end{align*}
\end{lemma}

We will prove Lemma \ref{lem5-3} at the end of this section.

To complete the proof of Theorem \ref{thm5-2}, we show the following lemma.

\begin{lemma}\label{lem5-2}
Fix $K\in\N$. For any $0\le t \le T$ and  $\delta>0$, we have
\begin{align*}
\varlimsup_{\ell\to\infty}\varlimsup_{\e\to0}\varlimsup_{\varepsilon\to0}\varlimsup_{N\to\infty}
\dfrac{1}{N}\log \bb P_\alpha^N\left( \left| \int_0^t V_1(\eta_s) ds \right| \ge\delta\right)  = -\infty ,
\end{align*}
where
\begin{align*}
V_1(\eta) \equiv  V_{1,N,\varepsilon,\e,\ell,K}(\eta) :=  h_K(\eta) - \dfrac{1}{\e N}\sum_{x=1}^{\e N} \overline H_{\ell,K}(\eta^{\varepsilon N}(x)) .
\end{align*}
\end{lemma}

The proof of Lemma \ref{lem5-2} is a consequence of the following propositions,
which are the local one-block estimate, local two-block estimate, and global replacement lemma, respectively.

\begin{proposition}\label{prop5-3}
Fix $K\in\N$. For any $0\le t \le T$ and  $\delta>0$, we have
\begin{align*}
\varlimsup_{\ell\to\infty}\varlimsup_{N\to\infty}
\dfrac{1}{N}\log \bb P_\alpha^N\left( \left| \int_0^t \left[ h_K(\eta_s) - H_{\ell,K}(\eta_s) \right] ds \right| \ge\delta\right)  = -\infty .
\end{align*}
\end{proposition}

\begin{proposition}\label{prop5-4}
Fix $K,\ell\in\N$. For any $0\le t \le T$ and  $\delta>0$, we have
\begin{align*}
\varlimsup_{\e\to0}\varlimsup_{N\to\infty}
\dfrac{1}{N}\log \bb P_\alpha^N
\left( \left| \int_0^t \left[ H_{\ell,K}(\eta_s) - \dfrac{1}{\e N}\sum_{x=1}^{\e N} H_{\ell,K}(\tau_x\eta_s)  \right] ds \right| \ge\delta\right)  = -\infty .
\end{align*}
\end{proposition}

\begin{proposition}\label{prop5-5}
Fix $K,\ell\in\N$ and $\e>0$. For any $0\le t \le T$ and  $\delta>0$, we have
\begin{align*}
\varlimsup_{\varepsilon\to0}\varlimsup_{N\to\infty}
\dfrac{1}{N}\log \bb P_\alpha^N
\left( \left| \int_0^t  \dfrac{1}{\e N}\sum_{x=1}^{\e N} \left[ H_{\ell,K}(\tau_x\eta_s) - \overline H_{\ell,K}(\eta_s^{\varepsilon N}(x)) \right] ds \right| \ge\delta\right)  = -\infty .
\end{align*}
\end{proposition}

Proposition \ref{prop5-5} is a classical replacement lemma and follows from \cite[Theorem 2]{bkl95}.
Propositions \ref{prop5-3} and \ref{prop5-4} are variants of replacement estimates
appearing in \cite{jls09} in the study of nonequilibrium fluctuations
for a tagged particle in the zero-range process.
Since \cite{jls09} is formulated for the environment process,
we briefly indicate the necessary adaptations below.
The boundedness of $h_K$ for fixed $K$ is essential in these arguments.
The main remaining difficulty in the present setting is
to control the error introduced by this cutoff and eventually let $K\to\infty$.

\begin{proof}[Proof of Proposition \ref{prop5-3}]
For each $a>0$, by the exponential Chebyshev inequality, we have
\begin{align*}
&\dfrac{1}{N}\log \bb P_\alpha^N\left( \left| \int_0^t \left[ h_K(\eta_s) - H_{\ell,K}(\eta_s) \right] ds \right| \ge\delta\right)\\
&\quad \le-a\delta + \dfrac{1}{N}\log \bb E_\alpha^N\left[ \exp \left\{ aN\left| \int_0^t \left[ h_K(\eta_s) - H_{\ell,K}(\eta_s) \right] ds \right| \right\} \right].
\end{align*}
We use the inequality $e^{|x|}\le e^x+e^{-x}$ to get rid of the absolute value and
we give a bound only for the positive term.
The same argument applies to the negative term. (This remark is used without further comment in the rest of this paper.)
By standard Feynman--Kac arguments, the problem is reduced to showing, for any $a>0$, 
\begin{align*}
\varlimsup_{\ell\to\infty}\varlimsup_{N\to\infty}\sup_f
\left\{\int_{X_N} W_{\ell,K}fd\nu_\alpha -(aN)^{-1}\int_{X_N} \sqrt f \left(-L_N\sqrt f\right) d\nu_\alpha \right\} = 0,
\end{align*}
where $W_{\ell,K}(\eta):=h_K(\eta)- H_{\ell,K}(\eta)$ 
and the supremum is taken over all densities $f$ with respect to $\nu_\alpha$: $\int fd\nu_\alpha=1$.
Since \(h\) is local,
we may assume that $\ell$ is sufficiently large
so that $\operatorname{supp} h\subset \Lambda_\ell$.
Moreover, by the same reasoning as in the proof of Lemma 6.3 of \cite{jls09},
the above supremum is bounded above by
\begin{align*}
\sup_{j\ge0}\sup_f\left\{\int_{\Sigma_{\Lambda_\ell, j}} W_{\ell,K}fd\nu_{\Lambda_\ell,j} -a^{-1}N\int_{\Sigma_{\Lambda_\ell, j}} \sqrt f \left(-L_{\Lambda_\ell}\sqrt f\right) d\nu_{\Lambda_\ell, j}\right\},
\end{align*}
where for each $\ell,j\in\N$, 
\begin{align*}
&\Lambda_\ell:=\{-\ell,\ldots, \ell\}, \quad \Sigma_{\Lambda_\ell, j}
:=\left\{ \eta\in\N^{\Lambda_\ell}: \sum_{x\in\Lambda_\ell}\eta(x)
=j \right\},\quad \nu_{\Lambda_\ell,j}(\cdot):=\nu_\alpha\left(\ \cdot \  | \Sigma_{\Lambda_\ell, j} \right),\\
&L_{\Lambda_\ell}f(\eta)
:=\dfrac12\sum_{\substack{x, y\in\Lambda_\ell \\ |x-y|=1}}
g(\eta(x))\left[f(\eta^{x,y})-f(\eta)\right], \quad f:\N^{\Lambda_\ell}\to\R,
\end{align*}
and the supremum in $f$ is taken over all densities $f$ with respect to $\nu_{\Lambda_\ell,j}$.
Moreover, in the definition of $\nu_{\Lambda_\ell,j}$,
the measure $\nu_\alpha$ is regarded as a probability measure on $\N^{\Lambda_\ell}$.

Let us decompose $W_{\ell,K}$ as
\begin{align*}
W_{\ell,K}(\eta)
&=\left(W_{\ell,K}(\eta) - \int_{\Sigma_{\Lambda_\ell, j}}
W_{\ell,K}d\nu_{\Lambda_\ell,j} \right)
+ \int_{\Sigma_{\Lambda_\ell, j}} W_{\ell,K}d\nu_{\Lambda_\ell,j}\\
&=: W_{\ell,j,K}(\eta)+\int_{\Sigma_{\Lambda_\ell, j}} W_{\ell,K}d\nu_{\Lambda_\ell,j}.
\end{align*}
Therefore, to conclude the proof, it is enough to show that for each $\ell$
\begin{align}\label{5-24}
\lim_{N\to\infty}\sup_{j\ge0}\sup_f
\left\{\int_{\Sigma_{\Lambda_\ell, j}}
W_{\ell,j,K}fd\nu_{\Lambda_\ell,j} -
a^{-1}N\int_{\Sigma_{\Lambda_\ell, j}} \sqrt f
\left(-L_{\Lambda_\ell}\sqrt f\right) d\nu_{\Lambda_\ell, j}\right\}=0,
\end{align}
and
\begin{align}\label{5-25}
\lim_{\ell\to\infty}\sup_{j\ge0}\left| \int_{\Sigma_{\Lambda_\ell, j}} W_{\ell,K}d\nu_{\Lambda_\ell,j} \right| = 0.
\end{align}

Let us turn to proving \eqref{5-24}.
Since $W_{\ell,j,K}$ has mean zero with respect to $\nu_{\Lambda_\ell,j}$ and $h_K$ is bounded ($h_K$ is supported only on finitely many configurations),
we can use \cite[Theorem A.3.1.1]{kl99}. Moreover, under the conditions \eqref{LG} and \eqref{M}, it has been established in \cite{lsv96} that
$L_{\Lambda_\ell}$ has a spectral gap on $\Sigma_{\Lambda_\ell, j}$ of order $\ell^{-2}$, uniformly in $j$.
Consequently, for sufficiently large $N$, we have
\begin{align*}
&\sup_{j\ge0}\sup_f\left\{\int_{\Sigma_{\Lambda_\ell, j}} W_{\ell,j,K}fd\nu_{\Lambda_\ell,j} -a^{-1}N\int_{\Sigma_{\Lambda_\ell, j}} \sqrt f \left(-L_{\Lambda_\ell}\sqrt f\right) d\nu_{\Lambda_\ell, j}\right\}\\
&\quad \le \dfrac{2a}{N} \sup_{j\ge0}\int_{\Sigma_{\Lambda_\ell, j}} W_{\ell,j,K} \left( -L_{\Lambda_\ell} \right)^{-1} W_{\ell,j,K} d\nu_{\Lambda_\ell,j} \le \dfrac{aC_0\ell^2}{N},
\end{align*}
where $C_0>0$ is a constant depending only on $h$ and $K$.
Thus, \eqref{5-24} is shown.

Finally, let us turn to proving \eqref{5-25}. Note that
\begin{align}\label{5-26}
\int_{\Sigma_{\Lambda_\ell, j}} W_{\ell,K}d\nu_{\Lambda_\ell,j} = \int_{\Sigma_{\Lambda_\ell, j}} h_Kd\nu_{\Lambda_\ell,j} - \int_{\N^{\Lambda_\ell}} h_K d\nu_{j/|\Lambda_\ell|}.
\end{align}
Let $\gamma>0$ and consider the case $(j/|\Lambda_\ell|)\le \gamma$.
Since $h_K$ is Lipschitz, we have
\begin{align*}
&\left|
\int_{\Sigma_{\Lambda_\ell,j}} h_Kd\nu_{\Lambda_\ell,j}
- \int_{\N^{\Lambda_\ell}} h_Kd\nu_{j/|\Lambda_\ell|}
\right| \\
&\quad \le
\int_{\Sigma_{\Lambda_\ell,j}}
|h_K-h_K(\mathfrak d)|d\nu_{\Lambda_\ell,j}\
+
\int_{\N^{\Lambda_\ell}}
|h_K-h_K(\mathfrak d)|d\nu_{j/|\Lambda_\ell|}\
\le C_0\gamma,
\end{align*}
for some constant $C_0>0$ depending only on $h$ and $K$.
Now consider the case $(j/|\Lambda_\ell|)\ge \gamma$. 
Since $h_K$ is bounded, by \cite[Corollary 6.1]{lsv96} and \eqref{2-5},
the absolute value of the expression in \eqref{5-26} is bounded above by
$C_1/(\ell\gamma)$ for some constant $C_1>0$ depending only on $h$ and $K$.
Therefore letting $\ell\uparrow\infty$ and $\gamma\downarrow0$ shows \eqref{5-25},
which completes the proof of Proposition \ref{prop5-3}.
\end{proof}

\begin{proof}[Proof of Proposition \ref{prop5-4}]
Since $H_{\ell,K}$ is bounded uniformly in $\ell$, to conclude the proposition it is enough to show
\begin{align*}
\varlimsup_{\e\to0}\varlimsup_{N\to\infty}
\dfrac{1}{N}\log \bb P_\alpha^N
\left( \left| \int_0^t \dfrac{1}{\e N}\sum_{x=2\ell+1}^{\e N}\left[ H_{\ell,K}(\eta_s) -  H_{\ell,K}(\tau_x\eta_s)  \right] ds \right| \ge\delta\right)  = -\infty,
\end{align*}
for any $\delta>0$.
Following the proof of Proposition \ref{prop5-3}, letting $W_{\ell,K,x}(\eta):=H_{\ell,K}(\eta)-H_{\ell,K}(\tau_x\eta)$,
we reduce the problem to showing, for any $a>0$,
\begin{align*}
\varlimsup_{\e\to0}\varlimsup_{N\to\infty}
\sup_{f,x}\left\{\int_{X_N} W_{\ell,K,x}fd\nu_\alpha -(aN)^{-1}\int_{X_N} \sqrt f \left(-L_N\sqrt f\right) d\nu_\alpha \right\} = 0,
\end{align*}
where the supremum is taken over all densities $f$ with respect to $\nu_\alpha$ and $x$ with $2\ell+1\le x \le \e N$.
Again by the same reasoning as in the proof of \cite[Lemma 6.5]{jls09},
the above supremum is bounded above by
\begin{align*}
\sup_{j\ge0}\sup_f\left\{\int_{\Sigma_{\Lambda_\ell, j}^*} \tilde W_{\ell,K}fd\nu_{\Lambda_\ell^*,j} -\dfrac{1}{2a\e}\int_{\Sigma_{\Lambda_\ell, j}^*} \sqrt f \left(-L_{\Lambda_\ell^*}\sqrt f\right) d\nu_{\Lambda_\ell^*, j}\right\},
\end{align*}
where for each $\ell,K,j\in\N$,
\begin{align*}
&\Lambda_\ell^*:=\{-\ell,\ldots, 3\ell+1\}, \quad \Sigma_{\Lambda_\ell, j}^*:=\left\{ \eta\in\N^{\Lambda_\ell^*}: \sum_{x\in\Lambda_\ell^*}\eta(x)=j \right\}, \quad\nu_{\Lambda_\ell^*,j}(\cdot):=\nu_\alpha\left(\ \cdot \ | \Sigma_{\Lambda_\ell, j}^*\right),\\
&\tilde W_{\ell,K}(\xi, \zeta):= H_{\ell,K}(\xi) - H_{\ell,K}(\zeta), \quad \eta=(\xi, \zeta)\in\N^{\Lambda_\ell^*}\simeq\N^{\Lambda_\ell}\times\N^{\Lambda_\ell},\\
&L_{\Lambda_\ell^*}f(\eta):=\dfrac12\sum_{\substack{x, y\in\Lambda_\ell^* \\ |x-y|=1}}g(\eta(x))\left[f(\eta^{x,y})-f(\eta)\right], \quad f:\N^{\Lambda_\ell^*}\to\R,
\end{align*}
and the supremum in $f$ is taken over all densities $f$ with respect to $\nu_{\Lambda_\ell^*,j}$.
Moreover, in the definition of $\nu_{\Lambda_\ell^*,j}$, the measure $\nu_\alpha$ is regarded as a probability measure on $\N^{\Lambda_\ell^*}$.
Since the above supremum is independent of $N$, to conclude the proof, 
it is enough to show that
\begin{align*}
\lim_{\e\to0}\sup_{j\ge0}\sup_f\left\{\int_{\Sigma_{\Lambda_\ell, j}^*} \tilde W_{\ell,K}fd\nu_{\Lambda_\ell^*,j} -\dfrac{1}{2a\e}\int_{\Sigma_{\Lambda_\ell, j}^*} \sqrt f \left(-L_{\Lambda_\ell^*}\sqrt f\right) d\nu_{\Lambda_\ell^*, j}\right\}=0,
\end{align*}
where the supremum is taken over all densities $f$ with respect to $\nu_{\Lambda_\ell^*, j}$.

Unlike in the one-block estimate, the expectation of $\tilde W_{\ell, K}$ with respect to $\nu_{\Lambda_\ell^*,j}$ vanishes,
so we can apply the Rayleigh estimate \cite[Theorem A.3.1.1]{kl99} at this stage (in our setting, an estimate corresponding to \cite[Lemma 6.6]{jls09} is not needed).
Since $L_{\Lambda_\ell^*}$ has a spectral gap on $\Sigma_{\Lambda_\ell, j}^*$ of order $\ell^{-2}$ uniformly in $j$,
by the Rayleigh estimate, we obtain for all sufficiently small $\e$
\begin{align*}
\sup_{j\ge0}\sup_f\left\{\int_{\Sigma_{\Lambda_\ell, j}^*} \tilde W_{\ell,K}fd\nu_{\Lambda_\ell^*,j} -\dfrac{1}{2a\e}\int_{\Sigma_{\Lambda_\ell, j}^*} \sqrt f \left(-L_{\Lambda_\ell^*}\sqrt f\right) d\nu_{\Lambda_\ell^*, j}\right\}
\le C_0a\ell^2\e,
\end{align*}
for some constant $C_0>0$ depending only on $h$ and $K$.
Letting $\e\to0$ completes the proof of Proposition \ref{prop5-4}.
\end{proof}

In order to prove Lemma \ref{lem5-3}, we shall replace $g(\eta(0))$
and $g_K(\eta(0))$ with averages as in Lemma \ref{lem5-2}, respectively.
By Lemma \ref{lem5-2}, $g_K(\eta(0))$ can be replaced with
\begin{align*}
\dfrac{1}{\e N}\sum_{x=1}^{\e N} \overline \Phi_{\ell,K}(\eta^{\varepsilon N}(x)) .
\end{align*}
On the other hand, since $g(\eta(0))$ is unbounded,
we cannot replace $g(\eta(0))$ by the same reasoning.
Instead, we use a change-of-variables argument to obtain the required replacement estimate.
The argument is based on ideas appearing in Lemma 3.1 of \cite{lan92} and Lemma 2.1 of \cite{ben96}.

\begin{lemma}\label{lem5-4}
For any $0\le t \le T$ and  $\delta>0$, we have
\begin{align}\label{5-1}
\varlimsup_{\e\to0}\varlimsup_{\varepsilon\to0}\varlimsup_{N\to\infty}
\dfrac{1}{N}\log \bb P_\alpha^N\left( \left| \int_0^t V_2(\eta_s) ds \right| \ge\delta\right)  = -\infty ,
\end{align}
where
\begin{align*}
V_2(\eta) \equiv  V_{2,N,\varepsilon,\e}(\eta) :=  g(\eta(0)) -\dfrac{1}{\e N}\sum_{x=1}^{\e N}  \Phi(\eta^{\varepsilon N}(x)) .
\end{align*}
\end{lemma}
\begin{proof}

To prove \eqref{5-1}, it is enough to show that,
for any $0\le t \le T$ and  $\delta>0$,
\begin{align}\label{5-2}
\varlimsup_{\e\to0}\varlimsup_{N\to\infty}
\dfrac{1}{N}\log \bb P_\alpha^N\left( \left| \int_0^t \widetilde V_2(\eta_s) ds \right| \ge\delta\right)  = -\infty ,
\end{align}
where
\begin{align*}
\widetilde V_2(\eta) \equiv  \widetilde V_{2,N,\e}(\eta) :=  g(\eta(0)) -\dfrac{1}{\e N}\sum_{x=1}^{\e N}  g(\eta(x)) .
\end{align*}
Indeed, \eqref{5-1} is a direct consequence of \eqref{5-2} together with 
the super-exponential replacement lemma established in \cite[Theorem 2]{bkl95}.

Let $C_0$ be the constant appearing in \eqref{2-1}
and set $\omega(\rho):=(2C_0)^{-1}\rho\log(1+\rho), \rho\ge0$.
Note that $\omega$ is convex and superlinear and satisfies
\begin{align}\label{5-14}
r :=  \log E_{\nu_\alpha}\left[\exp\left\{ \omega(g(\eta(0)))\right\} \right] <\infty,
\end{align}
by Stirling's formula.
We next show that
\begin{align}\label{5-3}
\varlimsup_{\e\to0}\varlimsup_{N\to\infty} \dfrac{1}{N}\log
\bb E_\alpha^N\left[\exp\left\{T^{-1}\e N\int_0^T \omega(g^{\e N}(\eta_s))ds\right\}\right]  \le  0 ,
\end{align}
where $g^{\e N}(\eta):=(\e N+1)^{-1}\sum_{x=0}^{\e N} g(\eta(x))$.
Indeed, by the invariance of $\nu_\alpha$ and the convexity of $\omega$, we have
\begin{align*}
&\bb E_\alpha^N\left[\exp\left\{T^{-1}\e N\int_0^T \omega(g^{\e N} (\eta_s))ds\right\}\right]\\
&\quad \le  E_{\nu_\alpha}\left[ \exp\left\{\e N\omega(g^{\e N}(\eta))\right\}\right]
\le  E_{\nu_\alpha}\left[ \exp\left\{\sum_{x=0}^{\e N}\omega(g(\eta(x)))\right\}\right]
=  e^{(\e N+1)r} .
\end{align*}
Thus, \eqref{5-3} has been shown.

Fix $a>0$. By the exponential Chebyshev inequality, we have
\begin{align*}
\dfrac{1}{N}\log \bb P_\alpha^N\left(  \int_0^t \widetilde V_2(\eta_s) ds \ge\delta\right)  \le  -a\delta +\dfrac{1}{N}\log 
\bb E_\alpha^N \left[ \exp \left\{ aN\int_0^t \widetilde V_2(\eta_s) ds \right\} \right]  .
\end{align*}
By the H\"older inequality, the second term of the right-hand side is bounded by
\begin{align*}
&\dfrac{1}{2N}\log \bb E_\alpha^N \left[ \exp \left\{ N\int_0^t 2a\widetilde V_2(\eta_s)-T^{-1}\e\omega(g^{\e N} (\eta_s)) ds \right\} \right]\\
&\quad+\dfrac{1}{2N}\log \bb E_\alpha^N \left[ \exp \left\{ T^{-1}\e N\int_0^t \omega(g^{\e N} (\eta_s)) ds \right\} \right] .
\end{align*}
Therefore, by \eqref{5-3} and standard Feynman--Kac arguments,
to conclude the proof it is enough to show that, for any $a>0$, the supremum
\begin{align}\label{5-4}
\sup_{f} \left \{ 2a\int \widetilde V_2 (\eta)f(\eta)\nu_\alpha(d\eta) - ND_N(f)
- T^{-1}\e\int  \omega(g^{\e N}(\eta))f(\eta)\nu_\alpha(d\eta) \right \} ,
\end{align}
vanishes as $N\to\infty$ and $\e\to0$,
where the supremum is taken over all densities $f$ with respect to $\nu_\alpha$
and $D_N(f)$ is the Dirichlet form given by
\begin{align*}
D_N(f):=\dfrac12\sum_{x\in\T_N}\int\left(\sqrt{f(\eta^{x,x+1})} - \sqrt{f(\eta)} \right)^2 g(\eta(x))\nu_\alpha(d\eta).
\end{align*}

We now claim that for any density $f$, we have
\begin{align}\label{5-5}
\int \widetilde V_2 (\eta)f(\eta)\nu_\alpha(d\eta) \le 
4\sqrt{\e N} \left(\int g^{\e N}(\eta)f(\eta)\nu_\alpha(d\eta)\right)^{1/2}D_N(f)^{1/2} .
\end{align}
To see this, for each $x=1,\dots, \e N$, we start with the change of variables $\eta^{0,x}\mapsto\xi$ to obtain
\begin{align}\label{5-22}
\int f(\eta^{0,x})g(\eta(0))\nu_\alpha(d\eta)
 =  \int f(\xi) g(\xi(x)) \nu_\alpha(d\xi) ,
\end{align}
and hence
\begin{align*}
\int \widetilde V_2 (\eta)f(\eta)\nu_\alpha(d\eta)
&= \dfrac{1}{\e N}\sum_{x=1}^{\e N}\int\left( f(\eta) -f(\eta^{0,x}) \right)g(\eta(0))\nu_\alpha(d\eta)\\
&= \dfrac{1}{\e N}\sum_{x=1}^{\e N}\sum_{0\le y \le x-1}\int\left( f(\eta^{0,y}) -f(\eta^{0,y+1}) \right)g(\eta(0))\nu_\alpha(d\eta)\\
&\le \sum_{x=0}^{\e N-1}\int\left| f(\eta^{0,x}) -f(\eta^{0,x+1}) \right|g(\eta(0))\nu_\alpha(d\eta),
\end{align*}
with the convention $\eta^{0,0}:=\eta$.
Using the decomposition
\begin{align*}
f(\eta^{0,y}) -f(\eta^{0,y+1}) = \left(\sqrt{f(\eta^{0,y})} - \sqrt{f(\eta^{0,y+1})} \right) \left(\sqrt{f(\eta^{0,y})} + \sqrt{f(\eta^{0,y+1})} \right),
\end{align*}
and the Schwarz inequality twice, we can bound the left-hand side of \eqref{5-5} by
\begin{align*}
&\left[ \sum_{x=0}^{\e N-1} \int \left(\sqrt{f(\eta^{0,x})} - \sqrt{f(\eta^{0,x+1})} \right)^2g(\eta(0))\nu_\alpha(d\eta) \right]^{1/2}\\
&\quad\times\left[ \sum_{x=0}^{\e N-1} \int \left(\sqrt{f(\eta^{0,x})} + \sqrt{f(\eta^{0,x+1})} \right)^2g(\eta(0))\nu_\alpha(d\eta) \right]^{1/2}.
\end{align*}
By \eqref{5-22}, the first line of the last expression is bounded above by $(2D_N(f))^{1/2}$.
For the second factor, again by \eqref{5-22}, we have
\begin{align*}
\left[2\sum_{x=0}^{\e N-1} \int \left(f(\eta^{0,x}) + f(\eta^{0,x+1}) \right)g(\eta(0))\nu_\alpha(d\eta) \right]^{1/2}
\le  \left[8\e N \int g^{\e N}(\eta)f(\eta)\nu_\alpha(d\eta) \right]^{1/2},
\end{align*}
for $N$ sufficiently large.
Thus, \eqref{5-5} has been shown.

On the other hand, by the convexity of $\omega$, we have
\begin{align}\label{5-6}
\int\omega(g^{\e N}(\eta))f(\eta)\nu_\alpha(d\eta)  \ge 
\omega\left( \int g^{\e N}(\eta)f(\eta)\nu_\alpha(d\eta) \right) .
\end{align}
From \eqref{5-5} and \eqref{5-6}, the supremum \eqref{5-4} can be bounded above by
\begin{align*}
\sup_{x\ge0, y\ge0} \left\{8a\sqrt{\e xy} - x -T^{-1}\e\omega(y) \right\}
=\e\sup_{y\ge0}\left\{ 16a^2y-\dfrac{1}{T}\omega(y)\right\},
\end{align*}
Therefore, the last expression vanishes as $\e\to0$,
which completes the proof of Lemma \ref{lem5-4}.
\end{proof}

We are now in the position to prove Lemma \ref{lem5-3}.

\begin{proof}[Proof of Lemma \ref{lem5-3}]
By translation invariance, it is enough to show the lemma for $x=0$.
Since $g-g_K\ge0$, it is enough to prove the assertion for $t=T$. Let
\begin{align*}
V_3(\eta) &\equiv  V_{3,N,\varepsilon,\e,\ell,K}(\eta) :=  g_K(\eta(0)) -\dfrac{1}{\e N}\sum_{x=1}^{\e N} \overline \Phi_{\ell,K}(\eta^{\varepsilon N}(x)) ,\\
V_4(\eta) &\equiv V_{4,N,\varepsilon,\e,\ell,K}(\eta) :=  \dfrac{1}{\e N}\sum_{x=1}^{\e N}  \left[ \Phi(\eta^{\varepsilon N}(x)) - \overline \Phi_{\ell,K}(\eta^{\varepsilon N}(x)) \right]  .
\end{align*}
Note that $g(\eta(0))-g_K(\eta(0))=V_2(\eta)-V_3(\eta)+V_4(\eta)$. By the union bound,
\begin{align*}
&\bb P_\alpha^N\left( \int_0^T \left[ g(\eta_s(0)) - g_K(\eta_s(0)) \right] ds \ge\delta\right)\\
& \le  \bb P_\alpha^N\left( \left| \int_0^T V_2(\eta_s) ds \right| \ge\delta/3\right)
+ \bb P_\alpha^N\left( \left| \int_0^T V_3(\eta_s) ds \right| \ge\delta/3\right)
+ \bb P_\alpha^N\left( \int_0^T V_4(\eta_s) ds \ge\delta/3\right) .
\end{align*}
Since Lemma \ref{lem5-4} holds for the case $h(\eta)=g(\eta(0))$,
the first term vanishes. The second term also vanishes thanks to Lemma \ref{lem5-2}.
Therefore, to conclude the proof of Lemma \ref{lem5-3},
it is enough to show that
\begin{align}\label{5-7}
\varlimsup_{K\to\infty}\varlimsup_{\ell\to\infty}\varlimsup_{\e\to0}\varlimsup_{\varepsilon\to0}\varlimsup_{N\to\infty}
\dfrac{1}{N}\log \bb P_\alpha^N\left( \int_0^T V_4(\eta_s)ds \ge\delta\right)  = -\infty .
\end{align}

Again by the exponential Chebyshev inequality, the problem is reduced to showing
\begin{align*}
\varlimsup_{K\to\infty}\varlimsup_{\ell\to\infty}\varlimsup_{\e\to0}\varlimsup_{\varepsilon\to0}\varlimsup_{N\to\infty}
\dfrac{1}{N}\log \bb E_\alpha^N\left[ \exp\left\{ aN \int_0^T V_4(\eta_s)ds \right\}\right]  = 0 ,
\end{align*}
for any $a>0$.
Recall the notation of Proposition \ref{prop5-6}.
For fixed $\varepsilon,\e,\ell,K$,
one can see that the functional $J_{\varepsilon,\e,\ell,K}$ is continuous on $\Omega_T$.
Moreover, the same discretization estimate,
together with the exponential moment bound below,
allows us to replace $\int_0^T V_4(\eta_s)\,ds$
by $J_{\varepsilon,\e,\ell,K}(\pi^N)$ in the logarithmic moment. Finally,
$$
|J_{\varepsilon,\e,\ell,K}(\pi^N)|
\le C_{\varepsilon,\e,\ell,K}\pi_0^N(\T),
$$
and hence the exponential tail condition for the unbounded version of Varadhan's lemma follows from the exponential moments of $\nu_\alpha$.
By Varadhan's lemma, 
\begin{align*}
\lim_{N\to\infty}\dfrac{1}{N}\log \bb E_\alpha^N\left[ \exp\left\{ aN \int_0^T V_4(\eta_s)ds \right\}\right]  =  
\sup_{\pi \in \Omega_T}\left\{ aJ_{\varepsilon,\e,\ell,K}(\pi) - I_\alpha(\pi) \right\}.
\end{align*}
Moreover, by Proposition \ref{prop5-6}, we have
\begin{align*}
\varlimsup_{\e\to0}\varlimsup_{\varepsilon\to0}\sup_{\pi\in \Omega_T}\left\{ aJ_{\varepsilon,\e,\ell,K}(\pi) - I_\alpha(\pi) \right\} 
=  \sup_{\pi \in \Omega_T^{I_\alpha}}\left\{ aJ_{\ell,K}(\pi) - I_\alpha(\pi) \right\} .
\end{align*}

Introducing the indicator functions $ {\bf 1}\{\rho(s,0)> A\}+ {\bf 1}\{\rho(s,0)\le A\}$
for each $A>0$
and noting $0\le\overline\Phi_{\ell,K}(\rho), \Phi(\rho)\le C_0\rho$ for any $\rho\ge0$,
we can bound $J_{\ell,K}(\pi)$ as
\begin{align*}
J_{\ell,K}(\pi) \le  C_0\int_0^T \rho(t,0) {\bf 1}\{\rho(t,0)\ge A\}dt
+ T\sup_{0\le\rho\le A} \left| R_{\ell,K}(\rho) \right| ,
\end{align*}
where $R_{\ell, K}(\rho):= \Phi(\rho) - \overline \Phi_{\ell,K}(\rho), \rho\ge0$.
From Proposition \ref{prop5-1} applied with $h(\eta)=g(\eta(0))$, the second term on the right-hand side vanishes as
$\ell\to\infty$ and $K\to\infty$ for each $A>0$.
Therefore, to conclude the proof it is enough to show that for any $a>0$
\begin{align*}
\lim_{A\to\infty}\sup_{\pi\in\Omega_T^{I_\alpha}}
\left\{ a\int_0^T \rho(t,0) {\bf 1}\{\rho(t,0)\ge A\}dt - I_\alpha(\pi) \right\}  =  0 .
\end{align*}
This is exactly the content of Proposition \ref{prop5-2}, which completes \eqref{5-7} and
the proof of Lemma \ref{lem5-3}.
\end{proof}

\section{Proof of the additive functional LDP}\label{sec:proof}

In this section, we prove Theorem \ref{thm5-3}.
Fix $0<\varepsilon<1/2$, $0<\e<1$ and $\ell,K\in\N$ and define the function $f_{N,\varepsilon, \e, \ell, K}:X_N\to\R$ by
\begin{align*}
f_{N,\varepsilon, \e, \ell, K}(\eta):=\dfrac{1}{\e N}\sum_{x=1}^{\e N} \overline H_{\ell,K}(\eta^{\varepsilon N}(x)), \quad \eta\in X_N.
\end{align*}
For fixed $\varepsilon,\e,\ell,K$, the corresponding functional of the empirical measure is continuous on $\Omega_T$ endowed with the Skorokhod $J_1$ topology. Indeed, this follows by approximating $\iota_\varepsilon$ in $L^1(\T)$ by continuous functions and using the Lipschitz continuity of $\overline H_{\ell,K}$.

Moreover, for fixed $\varepsilon,\e,\ell,K$, the difference between the discrete functional above and the corresponding functional of $\pi^N$ is bounded by
$$
\frac{C_{\varepsilon,\e,\ell,K}}{N}\bigl(1+\pi_0^N(\T)\bigr).
$$
Since the total mass is conserved and $\nu_\alpha$ has exponential moments of all orders, this difference is super-exponentially negligible at speed $N$.

\begin{proof}[Proof of the lower bound of Theorem \ref{thm5-3}]
Let $O$ be an open subset of $\R$ and fix $a\in O$.
To conclude the proof, it is enough to show that
\begin{align*}
\varliminf_{N\to\infty} \dfrac1N \log \bb P_\alpha^N\left( \Gamma_h^N \in O \right) \ge -\psi(a).
\end{align*}
Let $\delta>0$ be such that $[a-2\delta, a+2\delta]\subset O$. Then clearly, we have
\begin{align*}
\bb P_\alpha^N\left( \Gamma_h^N \in O \right) \ge \bb P_\alpha^N\left( \Gamma_h^N \in [a-2\delta, a+2\delta] \right).
\end{align*}
It follows from Theorem \ref{thm5-2} that
\begin{align*}
&\varliminf_{N\to\infty} \dfrac1N \log \bb P_\alpha^N\left( \Gamma_h^N \in O \right)\\
&\quad\ge \varlimsup_{\delta\to0}\varlimsup_{K\to\infty}\varlimsup_{\ell\to\infty}\varlimsup_{\e\to0} \varlimsup_{\varepsilon\to0}\varliminf_{N\to\infty} \dfrac1N \log \bb P_\alpha^N\left( \int_0^T f_{N,\varepsilon, \e, \ell, K}(\eta_t)dt \in (a-\delta, a+\delta) \right).
\end{align*}
By the contraction principle, for each fixed $\varepsilon, \e,\ell,K$,  we have
\begin{align*}
\varliminf_{N\to\infty} \dfrac1N \log \bb P_\alpha^N\left( \int_0^T f_{N,\varepsilon, \e, \ell, K}(\eta_t)dt \in (a-\delta, a+\delta) \right)
\ge - \inf_{b\in (a-\delta, a+\delta)} \psi_{\varepsilon, \e, \ell, K}(b).
\end{align*}
Therefore, the lower bound follows from Lemma \ref{lem5-11}.
\end{proof}

\begin{proof}[Proof of the upper bound of Theorem \ref{thm5-3}]
Let $C$ be a closed subset of $\R$. For each $\delta>0$, let
\begin{align*}
C^\delta:=\left\{a\in\R: \inf_{b\in C}|a -b|\le \delta \right\}.
\end{align*}
Note that
\begin{align*}
\bb P_\alpha^N\left( \Gamma_h^N \in C \right) \le 
\bb P_\alpha^N\left( \left| \int_0^T V(\eta_s)ds \right| \ge\delta \right) + \bb P_\alpha^N\left( \int_0^T f_{N,\varepsilon, \e, \ell, K}(\eta_t)dt \in C^\delta \right).
\end{align*}
The first term is super-exponentially small by Theorem \ref{thm5-2}.
On the other hand, for each fixed $\varepsilon,\e,\ell,K$,
again by the contraction principle, we have
\begin{align*}
\varlimsup_{N\to\infty} \dfrac1N \log \bb P_\alpha^N\left( \int_0^T f_{N,\varepsilon, \e, \ell, K}(\eta_t)dt \in C^\delta \right)
\le - \inf_{a\in C^\delta}\psi_{\varepsilon, \e, \ell, K}(a).
\end{align*}
Therefore, the upper bound follows from Lemma \ref{lem5-10}.
\end{proof}

\section*{Acknowledgment}
I. Seo has been supported by National Research Foundation of Korea funded by the Ministry of Science and ICT (RS-2025-23525546, RS-2023-NR076621, and RS-2026-25518141).
K. Tsunoda has been supported by JSPS KAKENHI Grant Number JP26K06832.

\section*{Use of generative AI}
Generative AI tools were used for language editing, typesetting assistance, and auxiliary checks of exposition and mathematical consistency. All mathematical ideas, arguments, and proof strategies in this work were developed by the authors, who take full responsibility for the final content.

\renewcommand{\bibliofont}{\fontsize{10}{10.6}\selectfont}

\end{document}